\documentclass[journal,twoside,web]{ieeecolor}

\usepackage{xr-hyper}
\usepackage{generic}
\usepackage{cite}
\usepackage{amsmath,amssymb,amsfonts}
\usepackage{algorithmic}
\usepackage{graphicx}
\usepackage{algorithm,algorithmic}
\usepackage{hyperref}
\hypersetup{hidelinks=true}

\usepackage{textcomp}
\def\BibTeX{{\rm B\kern-.05em{\sc i\kern-.025em b}\kern-.08em
		T\kern-.1667em\lower.7ex\hbox{E}\kern-.125emX}}

\usepackage{subcaption}

\usepackage{pdfpages}

\newtheorem{theorem}{Theorem}[section]

\newtheorem{corollary}{Corollary}[section]
\newtheorem{definition}{Definition}[section]
\newtheorem{remark}{Remark}[section]
\newtheorem{proposition}{Proposition}[section]

\begin{document}
\title{Value of Multiple-pursuer Single-evader Pursuit-evasion Game with Terminal Cost of Evader's Position: Relaxation of Convexity Condition}
\author{Weiwen Huang, Li Liang, Ningsheng Xu, and Fang Deng, \IEEEmembership{Senior Member, IEEE}
\thanks{(Corresponding author: Fang Deng.)}
\thanks{Weiwen Huang is with the School of Mathematics and Statistics, Beijing Institute of Technology, Beijing 100081, China, also with the National Key Lab of Autonomous Intelligent Unmanned Systems, Beijing 100081, China (e-mail: huangweiwen96@hotmail.com). }
\thanks{Li Liang is with the College of Information Science and Technology, Beijing University of Chemical Technology, Beijing 100029, China (e-mail: lianglibuct@buct.edu.cn).}
\thanks{Ningsheng Xu and Fang Deng are with the School of Automation, Beijing Institute of Technology, Beijing 100081, China, and with the National Key Lab of Autonomous Intelligent Unmanned Systems, Beijing 100081, China, and with Beijing Institute of Technology Chongqing Innovation Center, Chongqing 401120, China (e-mail:  Xuningsheng1@163.com; dengfang@bit.edu.cn).}
\thanks{This article has supplementary materials available on arxiv, provided by the authors.}
}

\maketitle

\begin{abstract}
In this study, we consider a multiple-pursuer single-evader quantitative pursuit-evasion game with payoff function that includes only the terminal cost. The terminal cost is a function related only to the terminal position of the evader. This problem has been extensively studied in target defense games. Here, we prove that a candidate for the value function generated by geometric method is the viscosity solution of the corresponding Hamilton-Jacobi-Isaacs partial differential equation (HJI PDE) Dirichlet problem. Therefore, the value function of the game at each point can be computed by a mathematical program. In our work, the convexity of the terminal cost or the target is not required. The terminal cost only needs to be locally Lipschitz continuous. The cases in which the terminal costs or the targets are not convex are covered. Therefore, our result is more universal than those of previous studies, and the complexity of the proof is improved. We also discuss the optimal strategies in this game and present an intuitive explanation of this value function.
\end{abstract}


\section{Introduction}
\label{sec:introduction}

Target defence games have been widely studied in recent years \cite{zhou2012general,chen2016multiplayer ,pachter2017differential,garcia2019strategies,garcia2020optimal,garcia2020multiple,dorothy2024one,selvakumar2019feedback,lee2021guarding,lee2024solutions,shishika2021partial,von2022circular,yan2020guarding,yan2022matching,deng2023multiple,yan2023homicidal,yan2024multiplayer,yan2025pursuit,fu2023justification,xu2024optimal,xu2024one}. In a target defense game, one team of players intends to attack target regions, while another team of players seeks to capture them to protect the target regions.

In this study, we consider a multiple-pursuer single-evader quantitative pursuit-evasion game. The cost function of the game includes only the terminal cost which is related only to the terminal position of the evader. This quantitative game has been widely studied in target defense games \cite{pachter2017differential,garcia2019strategies,yan2020guarding,yan2022matching,fu2023justification,lee2021guarding,lee2024solutions} where the terminal cost is often considered as a type of ``distance'' to the target.

A core problem in quantitative games is deriving the value function. Crandall and
Lions \cite{crandall1983viscosity} introduced the concept of viscosity solution for the first order partial
differential equations. This causes that the characterization
of the value function for differential games yields a satisfactory result. Typically, the value function is the unique viscosity solution to the corresponding HJI PDE Dirichlet problem \cite{lions1985differential,evans1984differential,soravia1993pursuit}. Solving the HJI PDE Dirichlet problem through numerical methods \cite{falcone2006numerical} is a way to obtain the value function. However, for multiple-pursuer single-evader games, this method suffers from the curse of dimensionality. As the number of players increases, the dimensionality of the state space increases. The scale of numerical solutions increases exponentially, causing difficulties in the numerical solution of the HJI PDE.

The geometric method is effective overcoming these difficulties. This method was first proposed by Isaacs in his pioneering work \cite{MR0210469}. He consider the Apollonius circle to be the safe region for the evader and used this concept to discuss the target defense problem. Recently, the geometric method has been widely used in researches on target defense games \cite{pachter2017differential,garcia2019strategies,yan2020guarding,yan2022matching,fu2023justification,lee2021guarding,lee2024solutions}. The researchers treated the games as quantitative games and obtained the candidates for the value function through the geometric method. They also attempted to verify this using HJI PDE. In \cite{pachter2017differential,garcia2019strategies,yan2020guarding,yan2022matching}, the terminal cost was assumed to be a specific function for deriving the closed form of the candidate. The results of these studies were not universal. In \cite{fu2023justification}, the authors no longer required a specific form of the terminal function, but their verification of the viscosity solution was limited to some differentiable points. This is not sufficient to indicate that the candidate function is the viscosity solution of the HJI PDE, because it is also possible that a continuous function which satisfies the HJI PDE at all differentiable points in the classical sense is not the viscosity solution (see \cite[Subsection I.3]{crandall1983viscosity}). The work \cite{lee2024solutions} is the latest in this field. In this study, the terminal cost is the distance from the evader's terminal position to an arbitrary convex target in Euclidean space, which is a continuously differentiable and convex function. Through convex analysis theory, the authors derived Clarke's generalized gradient of the candidate function as an outer estimate of its superdifferential and subdifferential. By this estimate, they demonstrated that the candidate function is the viscosity solution of the HJI PDE. This verification is more universal and complete than that in previous studies. However, their results also have some limitations. The cases in which the terminal costs or the targets are not convex are not considered in any of the works mentioned above, including \cite{lee2024solutions}.

The contributions of this study are as follows:

(i). To relax the conditions in \cite{lee2024solutions}, we prove that the candidate for the value function obtained by the geometric method is a viscosity solution of the corresponding HJI PDE Dirichlet problem without requiring the convexity of the terminal cost. The terminal cost only needs to be locally Lipschitz continuous. The approach stated in \cite{lee2024solutions} is not directly applicable to this scenario. We use Rockafellar's results \cite{rockafellar1982lagrange} regarding the differential properties of optimal value functions in nonlinear program to obtain the outer estimate of the Clarke's generalized gradient of the candidate function (\autoref{estimate of generalized gradient}), which is also an outer estimate of the superdifferential and the subdifferential. The relaxation of the conditions also causes the Clarke's generalized gradient of the candidate function to have a more complex structure. Therefore, the verification of the viscosity solution for the candidate function becomes more difficult. Thus, unlike the simple case in \cite{lee2024solutions}, our verification is divided into two parts: the verification of the viscosity subsolution (\autoref{th_sub2}) and the verification of the viscosity supersolution (\autoref{th_super}). Some geometric properties of the evader's dominance region, particularly \autoref{key_C-oval_property}, are used many times in this process. We summarize the verification in \autoref{th_sum_viscosity}. Our proof covers cases which the previous works have not covered.

(ii). Because the convexity of the terminal cost is not required, equilibrium state feedback strategies in the form of \cite{fu2023justification,yan2022matching,lee2024solutions} may not exist in the game. In Section \ref{sec_optimal_strategies}, we discuss the optimal strategy for this problem and provide an intuitive explanation of the value function from  a geometric method perspective.

The remainder of this paper is organized as follows. In Section \ref{sec_problem description}, we introduce some notations and terminologies and present the problem that we study. In Section \ref{sec_technical_preliminaries}, we introduce some technical preliminaries. In Section \ref{sec_verification of viscosity solution}, we verify the viscosity solution of the HJI PDE, which is the primary result of this study. In Section \ref{sec_optimal_strategies}, we discuss the optimal strategies for this game. In Section \ref{sec_application}, we show the application of the problem which is studied in this paper. Finally, we conclude the paper in Section \ref{sec_conclusion}.

\section{Problem description}\label{sec_problem description}
\subsection{Basic notation and concept}
\begin{table}[h]
	\caption{Notation}
	\label{notation_table}
	\centering
	\setlength{\tabcolsep}{3pt}
	\begin{tabular}{|p{50pt}|p{165pt}|}
		\hline
		Notation & Meaning\\
		\hline
		$\mathbb{R}^n$ & $n$-dimensional Euclidean space  \\
		$\mathbb{S}^n$ & $n$-sphere  \\
		$\left\| \cdot\right\| $ & Euclidean norm \\
		$x^\top$ & the transpose of vector $x$\\
		$B(x_0,r)$ & $\left\lbrace x\in\mathbb{R}^n:\left\| x-x_0\right\|<r \right\rbrace $ \\
		$D f(x)$ & the gradient of $f$ at $x$\\
		$D_y f(y,z)$ &the partial derivative of $f$ with respect to $y$ at $(y,z)$\\
		$\partial f(x)$ & the Clark's generalized gradient of $f$ at $x$ \cite{clarke1990optimization} \\
		$\partial_y f(y,z)$ & the Clark's generalized gradient of $f$  with respect to $y$ at $(y,z)$  \\
		$D^+ f(x)$ & the superdifferential of $f$ at $x$  \cite{bardi1997optimal}\\
		$D^- f(x)$ & the subdifferential of $f$ at $x$ \cite{bardi1997optimal} \\
		$\overline{A}$ & the closure of set $A$\\
		$A^\circ$ & the interior of set $A$\\
		$\partial A$ & the boundary of set $A$\\
		$\mathrm{co} A$ & the convex hull of set $A$\\
		$[n]$ & the set $\left\lbrace 1,2,\dots,n\right\rbrace$ \\
		\hline
	\end{tabular}
\end{table}
See Table.~\ref{notation_table} 
for the basic notation. 

The value function of a quantitative game may not be differentiable everywhere. In order to characterize the value function by the HJI PDE, the concept of viscosity solutions is applied. Below are some of the basic concepts about viscosity solutions.
\begin{definition}[Superdifferential and subdifferential \cite{bardi1997optimal}]
	Set $\Omega$ as an open domain in ${R}^n$, $f$ as a function from $\Omega$ to $\mathbb{R}$. Let $x\in \Omega$.
	The sets
	\begin{align*}
		&D^+f(x)\\
		&\triangleq 
		\left\lbrace 
		p\in\mathbb{R}^n:
		\limsup_{y\rightarrow x,y\in\Omega} 
		\frac
		{f(y)-f(x)-p^\top(y-x)}
		{\left\| y-x\right\| }
		\le 0
		\right\rbrace ,\\
		&D^-f(x)\\
		&\triangleq 
		\left\lbrace 
		p\in\mathbb{R}^n:
		\liminf_{y\rightarrow x,y\in\Omega} 
		\frac
		{f(y)-f(x)-p^\top(y-x)}
		{\left\| y-x\right\| }
		\ge 0
		\right\rbrace 
	\end{align*}
	are called the superdifferential and the subdifferential of $f$ at $x$, respectively. 
\end{definition}
\begin{remark}
	When $f$ is differentiable at $x$, $D^+f(x)=D^-f(x)=\left\lbrace Df(x)\right\rbrace $.
\end{remark}
\begin{definition}[Viscosity solution \cite{bardi1997optimal}]
	Let $\Omega$ be an open domain in ${R}^n$. Let $F=F(x,r,p)$ be a continuous function on $\Omega\times \mathbb{R}\times \mathbb{R}^n$. Consider the following partial differential equation:
	\begin{equation}\label{eq_PDE}
		F(x,f(x),Df(x))=0.
	\end{equation}
	A continuous function $f:\Omega\rightarrow\mathbb{R}$ is called a viscosity subsolution of the partial differential equation \eqref{eq_PDE} if 
	\begin{align*}
		F(x,f(x),p)\le0, \forall x\in \Omega,\ \forall p \in D^+f(x).
	\end{align*}
	A continuous function $f:\Omega\rightarrow\mathbb{R}$ is called a viscosity supersolution of the partial differential equation \eqref{eq_PDE} if 
	\begin{align*}
		F(x,f(x),p)\ge0, \forall x\in \Omega,\ \forall p \in D^-f(x).
	\end{align*}
	Finally, $f$ is called a viscosity solution of \eqref{eq_PDE} if it is simultaneously a viscosity subsolution and a viscosity supersolution. 
\end{definition}

The generalized gradients/differentials of some nonsmooth functions are also involved in our work. Below are some of the basic concepts about locally Lipschitz continuity and Clarke's generalized gradients.
\begin{definition}[Lipschitz continuous]
	Let $\Omega$ be a subset of $\mathbb{R}^n$. Let $L\ge 0$. A function $f:\Omega\rightarrow\mathbb{R}$ is said to be Lipschitz continuous (of rank $L$) on $\Omega$ if
	\begin{align*}
		\left| f(x_1)-f(x_2)\right|\le L\left\|x_1-x_2 \right\|,\ \forall x_1,x_2\in \Omega.  
	\end{align*}
\end{definition}

\begin{definition}[Locally Lipschitz continuous]
	Let $\Omega$ be a subset of $\mathbb{R}^n$. Let $L\ge 0$. Given a point $x\in \Omega$, a function $f:\Omega\rightarrow\mathbb{R}$ is said to be locally Lipschitz continuous (of rank $L$) near $x$ if there exist a constant $\epsilon>0$ such that $f$ is Lipschitz continuous (of rank $L$) on $B(x,\epsilon)\cap \Omega$. 
	$f$ is said to be be locally Lipschitz continuous on $\Omega$ if $f$ is locally Lipschitz continuous near any point of $\Omega$. 
\end{definition}

\begin{definition}[Clarke's generalized gradient \cite{clarke1990optimization}]\label{def_clarke_D}
	Let $\Omega$ be a open subset of $\mathbb{R}^n$. Let $x$ be a point of $\Omega$. The function $f:\Omega\rightarrow\mathbb{R}$ is locally Lipschitz continuous near $x$.   
	The set
	\begin{align*}
		\partial f (x)\triangleq\mathrm{co}\left\lbrace p\in\mathbb{R}^n:p=\lim_{n\rightarrow+\infty}Df(x_n),x_n\rightarrow x \right\rbrace 
	\end{align*}
	is called the Clarke's generalized gradient of $f$ at $x$.
\end{definition}

\begin{remark}\label{th_differentiable_Clarke}
	Rademacher's theorem \cite{evans1992measure} states that a  function which is Lipschitz continuous on an open subset of $\mathbb{R}^n$ is differentiable almost everywhere on that subset. Thus, \autoref{def_clarke_D} is well-defined.
\end{remark}
\begin{remark}
	The concept of Clarke's generalized gradients is compatible with that of classical gradients. When $f$ is differentiable at $x$, $\partial f(x)=\left\lbrace Df(x)\right\rbrace $.
\end{remark}

\begin{proposition}[2.1.2 proposition in page 27 of \cite{clarke1990optimization}]\label{th_Clark_GG_bounded}
	Let $\Omega$ be a subset of $\mathbb{R}^n$. Let $x$ be a point of $\Omega$. The function $f$ is locally Lipschitz continuous of rank $L$ near $x$. Then, for any $p\in \partial f(x)$, $\left\| p\right\| \le L$.
\end{proposition}

For any $s\in\mathbb{R}$, $A,B\subseteq \mathbb{R}^n$ and $c\in\mathbb{R}^n$, we denote $A+B\triangleq\left\lbrace a+b: a\in A ,b\in B\right\rbrace$, $sA\triangleq\left\lbrace sa : a\in A\right\rbrace$ and $c+A\triangleq\left\lbrace c\right\rbrace +A$.
\begin{proposition}[Corollary 2 in page 39 of \cite{clarke1990optimization}]\label{th_Clarke_linear_combination}
	Let $\Omega$ be a subset of $\mathbb{R}^n$. Let $x$ be a point of $\Omega$. The function $f_1,f_2$ are locally Lipschitz continuous near $x$. $s_1,s_2\in \mathbb{R}$. Then,
	\begin{align*}
		\partial\left( s_1 f_1+s_2 f_2\right) (x)\subseteq s_1\partial f_1(x)+s_2\partial f_2(x).
	\end{align*}
	The equality holds if at least one of $f_1$ and $f_2$ is continuously differentiable near $x$.
\end{proposition}
\begin{theorem}[Exercise 4.4 in Chapter II of \cite{bardi1997optimal}]\label{th_D+D-partial}
	Let $\Omega$ be a open domain in $\mathbb{R}^n$. Let $x\in\Omega$. The function $f:\Omega\rightarrow\mathbb{R}$ is locally Lipschitz continuous near $x$. Then,
	\begin{align*}
		D^+ f (x)\cup D^- f (x)\subseteq\partial f (x).
	\end{align*}
\end{theorem}
\subsection{Formulation of the game}
Consider a pursuit-evasion game taking place in an $n$-dimensional Euclidean space $\mathbb{R}^n (n\ge2)$. There are $m$ pursuers and one evader. We assume that all of them are mass points with Isaacs' simple motion\cite{MR0210469}. The maximum speed of all pursuers exceeds that of the evader. Without loss of generality, we assume that the maximum speed of the evader is $1$ and the maximum speed of the $i$th pursuer is $\alpha_i$ ($\alpha_i>1,\forall i\in[m]$). Let $l_i > 0$ be the capture radius of the $i$th pursuer. We assume that the distance between the $i$th pursuer and the evader at the initial moment is greater than $l_i$. When the distance between the $i$th pursuer and the evader is less than or equal to $l_i$, the $i$th pursuer captures the evader. Let $x_{P_1},\dots,x_{P_m},x_E\in\mathbb{R}^n$ be the location coordinates of the $m$ pursuers and the evader. The equations of motion are as follows:
\begin{equation}\label{dynamic system} 
	\begin{split}
		&\dot{x}_{P_i}(t)=\alpha_i u_{P_i}(t) , i=1,\dots,m,\\
		&\dot{x}_E (t)= u_E (t),  \\
		&x_{P_i}(0)=x_{P_i}^0, i=1,\dots,m,\\
		&x_E(0)=x_E^0,
	\end{split}
\end{equation}
where $x_{P_1}^0,\dots,x_{P_m}^0,x_E^0\in\mathbb{R}^n$  represent the initial positions and $u_{P_1}(\cdot),\dots,u_{P_m}(\cdot),u_E (\cdot)$ are the control input functions of the $m$ pursuers and the evader. $u_{P_1}(\cdot),\dots,u_{P_m}(\cdot),u_E (\cdot)$ are functions from the time interval $[0,+\infty)$ to $B(0,1)$.
Let 
\begin{align*}
	\mathcal{U}_{t_0} \triangleq \left\lbrace u:\mathbb[t_0,+\infty)\rightarrow B(0,1):u \text{ is Lebesgue measurable} \right\rbrace.
\end{align*}
The players' control input functions $u_{P_1}(\cdot),\dots,u_{P_m}(\cdot),u_E (\cdot)$  belong to $\mathcal{U}_0$. The $i$th pursuer's movement trajectory with the initial condition $x_{P_i}^0$ and the input function $u_{P_i}(\cdot)$ is denoted by $x_{P_i}(t;x_{P_i}^0,u_{P_i})$. The evader's movement trajectory with the initial condition $x_E^0$ and the input function $u_E(\cdot)$ is denoted by $x_E(t;x_E^0,u_E)$. For simplicity, $x_{P_i}(t;x_{P_i}^0,u_{P_i})$ and $x_E(t;x_E^0,u_E)$ are abbreviated without ambiguity as $x_{P_i}(t;u_{P_i})$ and $x_E(t;u_E)$ respectively. Let $\mathbf{y}=(x_{P_1}^\top,\dots,x_{P_m}^\top,x_E^\top)^\top$ and $\mathbf{p}=(p_{P_1}^\top,\dots,p_{P_m}^\top,p_E^\top)^\top$.
If one of the pursuers captures the evader, the game ends. The terminal set of the game is
\begin{align*}
	\mathcal{T}=\left\lbrace \mathbf{y}: \min_i( \left\| x_{P_i}-x_E\right\|-l_i )\le 0 \right\rbrace.  
\end{align*}
Its boundary is 
\begin{align*}
	\partial\mathcal{T}=\left\lbrace \mathbf{y}: \min_i( \left\| x_{P_i}-x_E\right\|-l_i )= 0 \right\rbrace.  
\end{align*}
The terminal time of the game is 
\begin{equation}
	\begin{aligned}
		&t_f(x_{P_1}^0,\dots,x_{P_m}^0,x_E^0,u_{P_1},\dots,u_{P_m},u_E) \\
		\triangleq& \inf \big\lbrace t : \min_i (\left\| x_{P_i}(t;x_{P_i}^0,u_{P_i})-x_E(t;x_E^0,u_E)\right\|-l_i) \le 0  \big\rbrace .
	\end{aligned}
\end{equation}
This can be abbreviated as $t_f(u_{P_1},\dots,u_{P_m},u_E)$ or $t_f$ when there is no ambiguity. We agree that the infimum of the empty set is $+\infty$. $t_f=+\infty$ implies that no pursuer can capture the evader with the initial condition and the control input functions.
Let
\begin{align*}
	&t_f^i(x_{P_i}^0,x_E^0,u_{P_i},u_E) \\
	\triangleq& \inf \left\lbrace t :\left\| x_{P_i}(t;x_{P_i}^0,u_{P_i})-x_E(t;x_E^0,u_E)\right\| \le l_i  \right\rbrace .
\end{align*}
This can also be abbreviated as $t_f^i(u_{P_i},u_E)$ or $t_f^i$ when no ambiguity exists. It is obvious that
\begin{align*}
	&t_f(x_{P_1}^0,\dots,x_{P_m}^0,x_E^0,u_{P_1},\dots,u_{P_m},u_E)\\
	=&\min_i t_f^i(x_{P_i}^0,x_E^0,u_{P_i},u_E).
\end{align*}
Let us now consider a quantitative game. $g$ is a locally Lipschitz continuous function from $\mathbb{R}^n$ to $\mathbb{R}$. We do not require the specific form of $g$. We define the Mayer type payoff function as follows,
\begin{align*}
	&J(x_{P_1}^0,\dots,x_{P_m}^0,x_E^0,u_{P_1},\dots,u_{P_m},u_E)\\
	\triangleq& g(x_E(t_f))\\
	=&g(x_E(t_f(x_{P_1}^0,\dots,x_{P_m}^0,x_E^0,u_{P_1},\dots,u_{P_m},u_E);x_E^0,u_E)),
\end{align*}
which the $m$ pursuers want to maximize and the evader wants to minimize. The terminal cost is related only to the evader's position. In target defense games \cite{pachter2017differential,garcia2019strategies,yan2020guarding,yan2022matching,fu2023justification,lee2021guarding,lee2024solutions}, $g$ is typically considered as a type of ``distance'' to the target.
The Hamiltonian \cite{soravia1993pursuit}\cite{bardi1997optimal} of this problem is as follows:
\begin{align*}
	&H(\mathbf{y},\mathbf{p})\\
	=&\inf_{u_{P_1},\dots,u_{P_m}\in B(0,1)} \sup_{u_E\in B(0,1)} (-\sum_{i=1}^{m}\alpha_i{p}_{P_i}^\top u_{P_i}-{p}_E^\top u_E)\\
	=&\sup_{u_E\in B(0,1)} \inf_{u_{P_1},\dots,u_{P_m}\in B(0,1)}  (-\sum_{i=1}^{m}\alpha_i{p}_{P_i}^\top u_{P_i}-{p}_E^\top u_E)\\
	=&-\sum_{i=1}^m\alpha_i\left\| {p}_{P_i}\right\| +\left\| {p}_E\right\| ,
\end{align*}
which satisfies the Isaacs condition \cite{MR0210469}.
Then we obtain the Hamilton-Jacobi-Isaacs partial differential equation (HJI PDE) Dirichlet problem of the quantitative game,
\begin{equation}\label{HJI}
	\begin{cases}
		H(\mathbf{y},DV(\mathbf{y}))=0, &\forall \mathbf{y} \in \Omega,\\
		V=g(x_E),  &\forall \mathbf{y} \in \partial\mathcal{T},
	\end{cases}
\end{equation}
where $\Omega=\mathbb{R}^{(m+1)n}\setminus\mathcal{T}$.
\subsection{Dominance region and candidate of value function}\label{subsec_Dominance region}
Let
\begin{align*}
	d_i(x;x_{P_i},x_{E})\triangleq -\left\| x-x_{P_i}\right\|+\alpha_i \left\| x-x_E \right\| + l_i.
\end{align*}
Given the position of $i$th pursuer and evader $x_{P_i},x_E$ which satisfies $\left\| x_{P_i}-x_E\right\|\ge l_i $, let 
\begin{align*}
	\mathcal{D}_i(x_{P_i},x_E)\triangleq \left\lbrace x\in\mathbb{R}^n:d_i(x;x_{P_i},x_{E}) \le 0\right\rbrace .
\end{align*}
It is easy to obtain: 
\begin{align*}
	\mathcal{D}_i(x_{P_i},x_E)^{\circ}= \left\lbrace x\in\mathbb{R}^n:d_i(x;x_{P_i},x_{E}) < 0\right\rbrace ,\\
	\partial\mathcal{D}_i(x_{P_i},x_E)= \left\lbrace x\in\mathbb{R}^n:d_i(x;x_{P_i},x_{E}) = 0\right\rbrace .
\end{align*}
$\mathcal{D}_i(x_{P_i},x_E)^{\circ}$ is called the dominance region of the evader relative to the $i$th pursuer.

Next, we present another presentation of $\mathcal{D}_i(x_{P_i},x_E)$, which was used in \cite{yan2022matching}\cite{yan2021optimal}. Pick $e\in\mathbb{S}^{n-1}$ and draw a ray from $x_E$ along $e$. We assume that there exists a intersection point between the ray and $\partial\mathcal{D}_i(x_{P_i},x_E)$, denoted by $x_E+\rho_i(x_{P_i},x_E,e) e$, $(\rho_i(x_{P_i},x_E,e)\ge0)$. Then,
\begin{align*}
	\left\| x_E+\rho_i e-x_{P_i}\right\|-\alpha_i \left\| \rho_i e \right\| = l_i.
\end{align*}
By solving the equation about $\rho_i$, we obtain only one non-negative solution:
\begin{equation}\label{rho}
	\begin{split}
		&\rho_i=\rho_i(x_{P_i},x_E,e)\\
		\triangleq&\frac{1}{\alpha_i^2-1} \bigg( -(\alpha_i l_i +e^\top(x_{P_i}-x_E))\\
		 &+\bigg((\alpha_i l_i +e^\top(x_{P_i}-x_E))^2\\
		 &+(\alpha_i^2-1)(\left\| x_{P_i}-x_E\right\|^2-l_i^2)\bigg)^{1/2}\bigg) .
	\end{split}
\end{equation}
The intersection point exists and is unique. The mapping $e\mapsto x_E+\rho_i(x_{P_i},x_E,e)e$ is a bijection from $\mathbb{S}^{n-1}$ to $\partial\mathcal{D}_i(x_{P_i},x_E)$ with inverse mapping $x\mapsto \frac{x-x_E}{\left\| x-x_E\right\| } $. Thus, 
\begin{equation}\label{sphere_homeomorphism}
	\begin{aligned}
		&\partial\mathcal{D}_i(x_{P_i},x_E)\\
		=& \left\lbrace x_E+\rho_i(x_{P_i},x_E,e)e:e\in\mathbb{S}^{n-1}\right\rbrace.
	\end{aligned}
\end{equation}
We can also obtain that
\begin{equation}\label{globe_homeomorphism}
	\begin{aligned}
		&\mathcal{D}_i(x_{P_i},x_E)\\
		=& \left\lbrace x_E+\rho e:e\in\mathbb{S}^{n-1},\rho\in[0,\rho_i(x_{P_i},x_E,e)]\right\rbrace.
	\end{aligned}
\end{equation}
Below are some of the properties of $\mathcal{D}_i(x_{P_i},x_E)$. The proofs are presented in the supplementary material.
\begin{proposition}\label{th C-oval boundedness}
	Assume that $\alpha_i>1$, $l_i\ge 0$ and $x_{P_i},x_E\in\mathbb{R}^n$ satisfy $\left\| x_{P_i}-x_E\right\| >l_i$. Then $\mathcal{D}_i(x_{P_i},x_E)$ is bounded.
\end{proposition}
\begin{proposition}\label{th_C-oval_convexity}
	Assume that $\alpha_i>1$, $l_i\ge 0$ and $x_{P_i},x_E\in\mathbb{R}^n$ satisfy $\left\| x_{P_i}-x_E\right\| >l_i$. Then $\mathcal{D}_i(x_{P_i},x_E)$ is strictly convex.
\end{proposition}

Let 
\begin{align*}
	\mathcal{D}^*(\mathbf{y})\triangleq\bigcap_i{\mathcal{D}_i(x_{P_i},x_E)}.
\end{align*}
It is easy to obtain the following representation,
\begin{align*}
	&\mathcal{D}^*(\mathbf{y})\\
	=&\left\lbrace x\in\mathbb{R}^n: \min_{i} d_i(x;x_{P_i},x_E) \le 0\right\rbrace\\
	=&\left\lbrace x_E+\rho e:e\in\mathbb{S}^{n-1},\rho\in[0,\min_i\rho_i(x_{P_i},x_E,e)]\right\rbrace ,\\
	&\partial \mathcal{D}^*(\mathbf{y})\\
	=&\left\lbrace x\in\mathbb{R}^n: \min_{i} d_i(x;x_{P_i},x_E) = 0\right\rbrace\\
	=&\left\lbrace x_E+\rho e:e\in\mathbb{S}^{n-1},\rho=\min_i\rho_i(x_{P_i},x_E,e)\right\rbrace.
\end{align*} 
We define a function
\begin{align*}
	V^g(\mathbf{y})\triangleq \min_{x\in{\mathcal{D}^*(\mathbf{y})}} g(x).
\end{align*}
$V^g$ is the candidate for the value function generated by geometric method.
\subsection{Goal of this paper}
In this study, we present a complete proof of that $V^g$ is a solution of \eqref{HJI} when $g$ is locally Lipschitz continuous. We do not require $g$ to be second-order continuously differentiable or convex, as in \cite{fu2023justification,lee2024solutions}. Since convex functions are locally Lipschitz continuous \cite{roberts1974another}, the cases in which the terminal costs or the targets are convex are special cases in our work. It is easy to obtain that $V^g$ satisfies the boundary condition of \eqref{HJI}. Thus, the main task is to prove that $V^g$ is a viscosity solution of the HJI PDE of \eqref{HJI}, i.e.\\
(a) $\forall \mathbf{y}\in\Omega, \forall\mathbf{p}=(p_{P_1}^\top,\dots,p_{P_m}^\top,p_E^\top)^\top \in D^+ V^g(\mathbf{y}),$
\begin{align*}
	-\sum_{i=1}^m\alpha_i\left\| {p}_{P_i}\right\| +\left\| {p}_E\right\| \le 0,
\end{align*}
which implies that $V^g$ is a viscosity subsolution of the HJI PDE of \eqref{HJI} in $\Omega$;\\
(b) $\forall \mathbf{y}\in\Omega, \forall\mathbf{p}=(p_{P_1}^\top,\dots,p_{P_m}^\top,p_E^\top)^\top\in D^- V^g(\mathbf{y}),$
\begin{align*}
	-\sum_{i=1}^m\alpha_i\left\| {p}_{P_i}\right\| +\left\| {p}_E\right\| \ge 0,
\end{align*}
which implies that $V^g$ is a viscosity supersolution of the HJI PDE of \eqref{HJI} in $\Omega$.
\section{Technical preliminaries}\label{sec_technical_preliminaries}

\subsection{An important property of dominance regions}\label{sec_key_C-oval_property}
\begin{proposition}\label{key_C-oval_property}
	Assume that $\alpha_i>1$, $l_i\ge 0$ and $x_{P_i},x_E\in\mathbb{R}^n$ satisfy $\left\| x_{P_i}-x_E\right\| >l_i$. Then $\forall x_1,x_2\in\partial\mathcal{D}_i(x_{P_i},x_E)$,
	\begin{align}\label{eq_key_C-oval_property}
		\frac{(x_1-x_{P_i})^\top(x_2-x_{P_i})}{\left\| x_1-x_{P_i}\right\| \left\| x_2-x_{P_i}\right\| }\ge \frac{(x_1-x_E)^\top(x_2-x_E)}{\left\| x_1-x_E\right\| \left\| x_2-x_E\right\| }.
	\end{align}
	The equality holds if and only if $x_1=x_2$.
\end{proposition}

\begin{figure}
	\centering
	\begin{subfigure}[b]{0.49\linewidth}
		\centering
		\includegraphics[width=\linewidth]{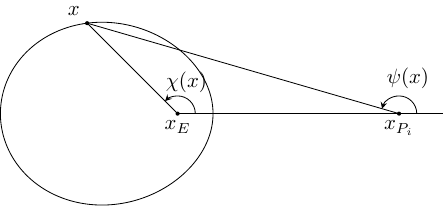}
		\caption{$\chi(x)<\pi$}
	\end{subfigure}
	\hfil
	\begin{subfigure}[b]{0.49\linewidth}
		\centering
		\includegraphics[width=\linewidth]{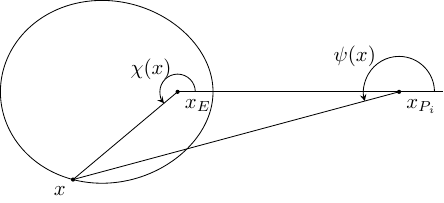}
		\caption{$\chi(x)>\pi$}
	\end{subfigure}
	\caption{$\chi(x)$ and $\psi(x)$ in the proof of \autoref{key_C-oval_property}.}
	\label{figure_chi_psi_1}
\end{figure}

\begin{proof}
	We first prove this conclusion holds when the dimension $n$ of the Euclidean space where the game happens is $2$.\\
	When $n=2$, $\partial \mathcal{D}_i(x_{P_i},x_E)$ is a Cartesian oval. Referring to Fig.~\ref{figure_chi_psi_1}, for any $x\in \partial \mathcal{D}_i(x_{P_i},x_E)$, let $\chi=\chi(x)$ denote the angle between vectors $x-x_E$ and  $x_{P_i}-x_E$ ($x_{P_i}-x_E$ is the starting edge, $x-x_E$ is the ending edge, and counterclockwise is the positive direction). Let $\psi=\psi(x)$ denote the angle between vectors $x-x_{P_i}$ and  $x_{P_i}-x_E$ ($x_{P_i}-x_E$ is the starting edge, $x-x_{P_i}$ is the ending edge, and counterclockwise is the positive direction). Without loss of generality, we assume that $\chi(x_2)-\chi(x_1)\in[0,\pi]$. If
	\begin{align}\label{c-oval_local_pro}
		\left| \frac{ d\psi(x)}{d\chi(x)}\right| < 1,
	\end{align}
	holds for any $x\in \partial \mathcal{D}_i(x_{P_i},x_E)$, we obtain
	\begin{subequations}
		\begin{align}
			&\left| \psi(x_2)-\psi(x_1)\right| \\
			=&\left| \int_{\chi(x_1)}^{\chi(x_2)} \frac{ d\psi(x)}{d\chi(x)} d\chi(x)\right| \\
			\le& \int_{\chi(x_1)}^{\chi(x_2)} \left| \frac{ d\psi(x)}{d\chi(x)} \right|d\chi(x) \\ 
			\le& \int_{\chi(x_1)}^{\chi(x_2)} d\chi(x) \label{c-equality}\\
			=& \chi(x_2)-\chi(x_1).
		\end{align}
	\end{subequations}
	Then,
	\begin{align*}
		\cos(\psi(x_2)-\psi(x_1)) \ge \cos(\chi(x_2)-\chi(x_1)).
	\end{align*}
	\eqref{eq_key_C-oval_property} holds by the geometric interpretation of the vectors' inner product. In particular, when $\chi(x_2)-\chi(x_1)\in(0,\pi]$, the equality of \eqref{c-equality} does not hold.
	We also obtain the conditions for the equality of \eqref{eq_key_C-oval_property}.
	Let us prove that \eqref{c-oval_local_pro} holds on the whole of $\partial \mathcal{D}_i(x_{P_i},x_E)$. Let
	$r_1(x)=\left\| x-x_{P_i}\right\|,
	r_2(x)=\left\| x-x_E\right\|,
	\sigma=\left\| x_{P_i}-x_E\right\|$.
	The equality $r_1=\alpha_i r_2+l_i$ holds. It is only necessary to consider the case $\chi\in[0,\pi]$ because of the symmetry of $\mathcal{D}_i$. By the parametric representation of the Cartesian oval in \cite[Lemma 1]{garcia2021cooperative2}, we obtain that $\chi$ and $r_2$ (or $r_1$) correspond one-to-one when $\chi\in[0,\pi]$ and the value range of $r_2$ is $[\frac{\sigma-l_i}{\alpha_i+1},\frac{\sigma-l_i}{\alpha_i-1}]$. By the law of cosines, 
	\begin{align*}
		\cos\chi=\frac{r_2^2+\sigma^2-r_1^2}{2r_2\sigma},\\
		\cos\psi=\frac{r_1^2+\sigma^2-r_2^2}{2r_1\sigma}.
	\end{align*}
	Treating $\cos\chi,\cos\psi, r_1$ as functions of $r_2$, we take the derivative of the above equations on both sides with respect to $r_2$:
	\begin{align*}
		\frac{d\cos\chi}{d r_2}&=\frac{r_1^2+r_2^2-2\alpha_i r_1 r_2-\sigma^2}{2\sigma r_2^2},\\
		\frac{d\cos\psi}{d r_2}&=\frac{\alpha_i r_1^2+\alpha_i r_2^2-2 r_1 r_2-\alpha_i \sigma^2}{2\sigma r_1^2}.
	\end{align*}
	By the above equations and the law of sines,
	\begin{align*}
		\frac{d\psi}{d\chi}=&\frac{\sin\chi}{\sin\psi} \cdot \frac{d\cos\psi}{d\cos\chi}\\
		=&\frac{r_1}{r_2}\cdot \frac{d\cos \psi /dr_2}{d\cos \chi /dr_2}\\
		=&\frac{\alpha_i r_1^2r_2-2r_1r_2^2+\alpha_i r_2^3-\alpha_i\sigma^2r_2}{r_1^3-2\alpha_i r_1^2r_2+r_1r_2^2-\sigma^2r_1}.
	\end{align*}
	For convenience, let
	\begin{align}
		f_1&=\alpha_i r_1^2r_2-2r_1r_2^2+\alpha_i r_2^3-\alpha_i\sigma^2r_2,\\
		f_2&=r_1^3-2\alpha_i r_1^2r_2+r_1r_2^2-\sigma^2r_1 \label{f1}.
	\end{align}
	Substituting $r_2=\frac{1}{\alpha_i}r_1-\frac{l}{\alpha_i}$ into \eqref{f1}, we obtain
	\begin{align*}
		f_2&=r_1\left[ \left(\frac{1}{\alpha_i^2}-1\right)r_1^2+2l_i\left(1-\frac{1}{\alpha_i^2}\right)r_1+\frac{l_i^2}{\alpha_i^2}-\sigma^2\right] \\
		&=-\frac{r_1}{\alpha_i^2}\left[ \left( \alpha_i^2-1\right) \left( r_1-l_i\right) ^2 +\alpha_i^2 \left( \sigma ^2- l_i^2 \right) \right]. 
	\end{align*}
	Due to $\alpha_i >1,\sigma>l_i$, we obtain $\forall r_1>0,f_2(r_1)<0$. Subsequently,
	\begin{align*}
		\left| \frac{f_1}{f_2}\right|< 1,\forall r_2\in \left[ \frac{\sigma-l_i}{\alpha_i+1},\frac{\sigma-l_i}{\alpha_i-1}\right]
	\end{align*}
	is converted into
	\begin{align}\label{converted}
		f_1-f_2> 0 \bigwedge f_1+f_2< 0 ,\forall r_2\in \left[ \frac{\sigma-l_i}{\alpha_i+1},\frac{\sigma-l_i}{\alpha_i-1}\right].
	\end{align}
	Let $g_1=f_1+f_2,\:g_2=f_1-f_2$. Substitute $r_1=\alpha_i r_2+l_i$ into $g_1,g_2$ and simplify them:
	\begin{align*}
		g_1(r_2)=l_i(\alpha_i^2-1)r_2^2-2\alpha_i (\sigma^2-l_i^2)r_2-l_i(\sigma^2-l_i^2),\\
		g_2(r_2)=2(\alpha_i^3-\alpha_i)r_2^3+3l_i(\alpha_i^2-1)r_2^2+l_i(\sigma^2-l_i^2).
	\end{align*}
	Consider $g_1$,
	\begin{align*}
		&g_1\left( \frac{\sigma-l_i}{\alpha_i+1}\right)\\ =&\frac{-((\alpha_i+1)l_i+2\alpha_i\sigma)(\sigma-l_i)-l_i(\alpha_i+1)(\sigma+l_i)}{(\alpha_i+1)/(\sigma-l_i)}<0,\\
		&g_1\left( \frac{\sigma-l_i}{\alpha_i-1}\right)\\ =&\frac{-((\alpha_i-1)l_i+2\alpha_i\sigma)(\sigma-l_i)-l_i(\alpha_i-1)(\sigma+l_i)}{(\alpha_i-1)/(\sigma-l_i)}<0.
	\end{align*}
	$g_1(r_2)$ is a quadratic function with a positive quadratic coefficient. Thus, it is a convex function. Then, $g_1(r_2)<0,\forall r_2\in \left[ \frac{\sigma-l_i}{\alpha_i+1},\frac{\sigma-l_i}{\alpha_i-1}\right]$. The cubic coefficient of $g_2(r_2)$ is positive, and the other coefficients of $g_2(r_2)$ are non-negative. Thus,  $g_2(r_2)>0,\forall r_2\in \left[ \frac{\sigma-l_i}{\alpha_i+1},\frac{\sigma-l_i}{\alpha_i-1}\right]$. \eqref{converted} holds. We complete the proof of the case $n=2$. \\
	Next, we consider the case $n\ge3$.\\
	Select $x_1,x_2 \in \partial \mathcal{D}_i(x_{P_i},x_E)$. Let
	\begin{align}\label{pro1}
		e_k=\frac{x_k-x_E}{\left\| x_k-x_E\right\| },\ k=1,2.
	\end{align}
	For convenience, let
	\begin{align*}
		\rho_i^{(k)}=\rho_i(x_{P_i},x_E,e_k),\ k=1,2.
	\end{align*}
	According to Subsection \ref{subsec_Dominance region}, we have
	\begin{equation}\label{pro2}
		\begin{split}
			&x_k=x_E+\rho_i^{(k)}e_k,\ k=1,2,\\
			&\left\| x_k-x_E\right\|=\rho_i^{(k)},\ k=1,2,\\
			&\left\| x_k-x_{P_i}\right\|=\alpha_i\rho_i^{(k)}+l_i,\ k=1,2.
		\end{split}
	\end{equation}
	Establish a new Cartesian coordinate system in $\mathbb{R}^n$ with the origin of the original coordinate system as the new origin and the vector $x_{P_i}-x_E$ as the positive direction of the new first coordinate axis.
	The coordinate of $x_{P_i}-x_E$ in the new system is $(\left\|x_{P_i}-x_E \right\|,0,\dots,0)^\top$. By the transformation between Cartesian coordinate system and spherical coordinate system \cite{blumenson1960derivation}, $e_1,e_2$ can be parameterized in the new coordinate system as follows,
	\begin{equation}\label{e1e2_parametric_representation}
		\begin{split}
			e_1\mapsto 
			\begin{pmatrix}
				\cos \eta_1\\
				\sin \eta_1 \cos \eta_2\\
				\sin \eta_1 \sin \eta_2 \cos \eta_3\\
				\vdots\\
				\sin \eta_1 \dots \sin \eta_{n-2} \cos\eta_{n-1}\\
				\sin \eta_1 \dots \sin \eta_{n-2} \sin\eta_{n-1}\\
			\end{pmatrix},
			\\
			e_2\mapsto 
			\begin{pmatrix}
				\cos \zeta_1\\
				\sin \zeta_1 \cos \zeta_2\\
				\sin \zeta_1 \sin \zeta_2 \cos \zeta_3\\
				\vdots\\
				\sin \zeta_1 \dots \sin \zeta_{n-2} \cos\zeta_{n-1}\\
				\sin \zeta_1 \dots \sin \zeta_{n-2} \sin\zeta_{n-1}\\
			\end{pmatrix}.
		\end{split}
	\end{equation}
	where $\eta_j,\zeta_j \in [0,\pi], \forall j\in [n-2]$ and $\eta_{n-1},\zeta_{n-1} \in [0,2\pi)$. Then, we have
	\begin{align}\label{pro3}
		e_1^\top(x_{P_i}-x_E)=\left\|x_{P_i}-x_E \right\|\cos\eta_1,\\
		e_2^\top(x_{P_i}-x_E)=\left\|x_{P_i}-x_E \right\|\cos\zeta_1.
	\end{align}
	Let
	\begin{align*}
		&\bar{x}_{P_i}=(\left\| x_{P_i}-x_E\right\| ,0)^\top,\\
		&\bar{x}_E=(0,0)^\top,\\
		&\bar{e}_1=(\cos\eta_1,\sin\eta_1)^\top,\\
		&\bar{e}_2=(\cos\zeta_1,\sin\zeta_1)^\top,
	\end{align*}
	It is easy to obtain that
	\begin{equation}\label{pro4}
		\begin{split}
			&\left\| \bar{x}_{P_i}-\bar{x}_E\right\| =\left\| x_{P_i}-x_E\right\|,\\
			&\bar{e}_k^\top (\bar{x}_{P_i}-\bar{x}_E)=e_k^\top (x_{P_i}-x_E),\ k=1,2.
		\end{split}
	\end{equation}
	Let
	\begin{subequations}
		\begin{align}
			&\bar{\rho}_i^{(k)}=\rho_i(\bar{x}_{P_i},\bar{x}_E,\bar{e}_k),\ k=1,2,\\
			&\bar{x}_k=\bar{x}_E+\bar{\rho}_i^{(k)}\bar{e}_k,\ k=1,2.\label{pro5}
		\end{align}
	\end{subequations}
	According to Subsection \ref{subsec_Dominance region}, we have $\bar{x}_1,\bar{x}_2\in\partial\mathcal{D}_i(\bar{x}_{P_i},\bar{x}_E)$ and 
	\begin{equation}\label{pro6}
		\begin{split}
			&\left\| \bar{x}_k-\bar{x}_E\right\|=\bar{\rho}_i^{(k)},\ k=1,2,\\
			&\left\| \bar{x}_k-\bar{x}_{P_i}\right\|=\alpha_i\bar{\rho}_i^{(k)}+l_i,\ k=1,2.
		\end{split}
	\end{equation}
	According to the case $n=2$ of \autoref{key_C-oval_property} (which we have proven earlier), we have
	\begin{align}\label{pro7}
		\frac{(\bar{x}_1-\bar{x}_{P_i})^\top(\bar{x}_2-\bar{x}_{P_i})}{\left\| \bar{x}_1-\bar{x}_{P_i}\right\| \left\| \bar{x}_2-\bar{x}_{P_i}\right\| }\ge \frac{(\bar{x}_1-\bar{x}_E)^\top(\bar{x}_2-\bar{x}_E)}{\left\| \bar{x}_1-\bar{x}_E\right\| \left\| \bar{x}_2-\bar{x}_E\right\| }.
	\end{align}
	Substituting \eqref{pro5}\eqref{pro6} into \eqref{pro7}, we obtain
	\begin{equation}\label{pro8}
		\begin{split}
			\frac{\left( \bar{x}_E+\bar{\rho}_i^{(1)}\bar{e}_1-\bar{x}_{P_i}\right) ^\top\left( \bar{x}_E+\bar{\rho}_i^{(2)}\bar{e}_2-\bar{x}_{P_i}\right) }
			{\left( \alpha_i\bar{\rho}_i^{(1)}+l_i\right) \left( \alpha_i\bar{\rho}_i^{(2)}+l_i\right) }
			\ge \bar{e}_1^\top \bar{e}_2.
		\end{split}
	\end{equation}
	By transforming and reorganizing the inequality of \eqref{pro8}, we obtain  
	\begin{equation}\label{pro9}
		\begin{split}
			&\frac{(\bar{x}_{P_i}-\bar{x}_E)^2-\bar{\rho}_i^{(1)}\bar{e}_1^\top(\bar{x}_{P_i}-\bar{x}_E)-\bar{\rho}_i^{(2)}\bar{e}_2^\top(\bar{x}_{P_i}-\bar{x}_E)}
			{\left( \alpha_i\bar{\rho}_i^{(1)}+l_i\right) \left( \alpha_i\bar{\rho}_i^{(2)}+l_i\right)-\bar{\rho}_i^{(1)}\bar{\rho}_i^{(2)}}\\
			&\ge \bar{e}_1^\top \bar{e}_2.
		\end{split}
	\end{equation}
	Consider $e_1^\top e_2$,
	\begin{align*}
		&e_1^\top e_2\\
		=&\cos \eta_1 \cos\zeta_1 
		+\sum_{k=2}^{n-1} \left( \prod_{j=1}^{k-1}\left( \sin\eta_j \sin \zeta_j\right) \right) \cos \eta_k \cos \zeta_k \\
		&+\prod_{j=1}^{n-1}\left( \sin\eta_j \sin \zeta_j\right) \\
		=&\cos \eta_1 \cos\zeta_1 
		+\sum_{k=2}^{n-2} \left( \prod_{j=1}^{k-1}\left( \sin\eta_j \sin \zeta_j\right) \right) \cos \eta_k \cos \zeta_k \\
		&+\prod_{j=1}^{n-2}\left( \sin\eta_j \sin \zeta_j\right) \cos(\eta_{n-1}-\zeta_{n-1}) \\
		\le &\cos \eta_1 \cos\zeta_1 +\sum_{k=2}^{n-2} \left( \prod_{j=1}^{k-1}\left( \sin\eta_j \sin \zeta_j\right) \right) \cos \eta_k \cos \zeta_k \\ &+\prod_{j=1}^{n-2}\left( \sin\eta_j \sin \zeta_j\right) \\
		&\cdots\cdots \\
		\le &\cos \eta_1 \cos\zeta_1 
		+ \sin\eta_1 \sin \zeta_1  \cos \eta_2 \cos \zeta_2 \\
		&+ \sin\eta_1 \sin \zeta_1  \sin \eta_2 \sin \zeta_2 \\
		=&\cos \eta_1 \cos\zeta_1 
		+ \sin\eta_1 \sin \zeta_1 \cos(\eta_2-\zeta_2) \\
		\le & \cos \eta_1 \cos\zeta_1 
		+ \sin\eta_1 \sin \zeta_1 \\
		=& \bar{e}_1^\top\bar{e}_2,
	\end{align*}
	By above inequality, we have
	\begin{align}\label{pro10}
		e_1^\top e_2\le \bar{e}_1^\top\bar{e}_2.
	\end{align}
	The equality holds if and only if $\eta_1 \zeta_1=0$ or the vector
	\begin{align*}
		\begin{pmatrix}
			\cos \eta_2\\
			\sin \eta_2 \cos \eta_3\\
			\sin \eta_2 \sin \eta_3 \cos \eta_4\\
			\vdots\\
			\sin \eta_2 \dots \sin \eta_{n-2} \cos\eta_{n-1}\\
			\sin \eta_2 \dots \sin \eta_{n-2} \sin\eta_{n-1}\\
		\end{pmatrix}
	\end{align*}
	is equal to the vector 
	\begin{align*}
		\begin{pmatrix}
			\cos \zeta_2\\
			\sin \zeta_2 \cos \zeta_3\\
			\sin \zeta_2 \sin \zeta_3 \cos \zeta_4\\
			\vdots\\
			\sin \zeta_2 \dots \sin \zeta_{n-2} \cos\zeta_{n-1}\\
			\sin \zeta_2 \dots \sin \zeta_{n-2} \sin\zeta_{n-1}\\
		\end{pmatrix}.
	\end{align*}

	Recalling \eqref{rho}, we obtain that $\rho_i(x_{P_i},x_E,e)$ depends only on $e^\top (x_{P_i}-x_E)$ and $\left\| x_{P_i}-x_E\right\| $ for $x_{P_i}$, $x_E$ and $e$. Thus, by \eqref{pro4} 
	\begin{align}\label{pro11}
		\bar{\rho}_i^{(k)}=\rho_i^{(k)},\ k=1,2.
	\end{align}
	
	Substituting \eqref{pro4}\eqref{pro10}\eqref{pro11} into \eqref{pro9}, we have
	\begin{equation}\label{pro12}
		\begin{split}
			&\frac{(x_{P_i}-x_E)^2-\rho_i^{(1)}e_1^\top(x_{P_i}-x_E)-\rho_i^{(2)}e_2^\top(x_{P_i}-x_E)}
			{\left( \alpha_i\rho_i^{(1)}+l_i\right) \left( \alpha_i\rho_i^{(2)}+l_i\right)-\rho_i^{(1)}\rho_i^{(2)}}\\
			&\ge e_1^\top e_2.
		\end{split}
	\end{equation}
	By transforming and reorganizing the inequality of \eqref{pro12}, we obtain
	\begin{equation}\label{pro13}
		\begin{split}
			\frac{\left( x_E+\rho_i^{(1)}e_1-x_{P_i}\right) ^\top\left( x_E+\rho_i^{(2)}e_2-x_{P_i}\right) }
			{\left( \alpha_i\rho_i^{(1)}+l_i\right) \left( \alpha_i\rho_i^{(2)}+l_i\right) }
			\ge e_1^\top e_2.
		\end{split}
	\end{equation}
	Substituting \eqref{pro1}\eqref{pro2} into \eqref{pro13}, we have
	\begin{align}\label{pro14}
		\frac{(x_1-x_{P_i})^\top(x_2-x_{P_i})}{\left\| x_1-x_{P_i}\right\| \left\| x_2-x_{P_i}\right\| }\ge \frac{(x_1-x_E)^\top(x_2-x_E)}{\left\| x_1-x_E\right\| \left\| x_2-x_E\right\| }.
	\end{align}
	If the equality of \eqref{pro14} holds, according to the equality conditions of \eqref{pro8} and \eqref{pro10}, $e_1=e_2$. This implies $x_1=x_2$. The proof is completed.  
\end{proof}
\subsection{Pursuit strategies generated by dominance regions}\label{sec_pursuit_strategies}
This part aims to present \autoref{th_epsilon_lemma} which is used in the proof of \autoref{th_super}. Further discussion on the pursuit strategy is presented in Section \ref{sec_optimal_strategies}.

First, we present a pursuit strategy generated by $\partial\mathcal{D}_i(x_{P_i},x_E)$ for the $i$th pursuer. For this strategy, the pursuer decides the current control input by knowing current states and the evader’s current control input.
\par
Given the current position $x_{P_i},x_E\in\mathbb{R}^n$ and the evader's current input $u_E\in B(0,1)$, assume that $\left\| x_{P_i}-x_E\right\| > l_i $. The current input for the $i$th pursuer is determined as follows,
\begin{equation}
	\begin{aligned}\label{eq_prusuer_current_input}
		&u_{P_i}=\gamma_{P_i}(x_{P_i},x_E,u_E)\\
		\triangleq& 
		\begin{cases}
			\frac{x_E+\rho_i\left( x_{P_i},x_E,\frac{u_E}{\left\| u_E\right\| }\right) \frac{u_E}{\left\| u_E\right\| }-x_{P_i}}{\left\| x_E+\rho_i\left( x_{P_i},x_E,\frac{u_E}{\left\| u_E\right\| }\right) \frac{u_E}{\left\| u_E\right\| }-x_{P_i}\right\| } & 0<\left\| u_E\right\| \le1\\
			\frac{x_E-x_{P_i}}{\left\| x_E-x_{P_i}\right\| } & u_E=0
		\end{cases}.
	\end{aligned}
\end{equation}
When $0<\left\| u_E\right\| \le1$, the $i$th pursuer moves toward the point on $\partial\mathcal{D}_i(x_{P_i},x_E)$ toward which the evader moves. When $u_E= 0$, the $i$th pursuer moves toward the evader.
\begin{proposition}\label{th_pursuit_strategy}
	Let $\alpha_i>1,\ l_i\ge 0$. Assume that $x_{P_i}(t),x_E(t)$ are piecewise smooth in $[a,b)$ and satisfy
	\begin{align*}
		\left\| x_{P_i}(t)-x_E(t)\right\| > l_i ,\\
		\left\| \dot{x}_{E}(t) \right\|  \le 1,
	\end{align*} 
	and
	\begin{equation}\label{eq_ode_P-strategy}
		\begin{split}
			\dot{x}_{P_i}(t)&=\alpha_i\gamma_{P_i}(x_{P_i}(t),x_E(t),u_E(t)),\\
			\dot{x}_{E}(t)&=u_E(t),
		\end{split}
	\end{equation}
	for any $t\in[a,b)$. Then,\\
	(a) for any $t\in[a,b)$, $\frac{d}{dt}\left\| x_{P_i}(t)-x_E(t)\right\| < 1-\alpha_i$,\\
	(b) for any $t_1 \in[a,b)$, for any $t_2 \in (t_1,b)$, $\mathcal{D}_i(x_{P_i}(t_2),x_E(t_2)) \subseteq \mathcal{D}_i(x_{P_i}(t_1),x_E(t_1))$.
\end{proposition}
In \cite{dutkevich1972games}, the authors obtained a similar result when the game takes place in a Euclidean plane. We present a more general proof in the supplementary material. Below is a corollary of \autoref{th_pursuit_strategy}.
\begin{corollary}\label{th_epsilon_lemma}
	Let $\alpha_i>1,\ l_i\ge 0$. Assume that $x_{P_i},x_{E}\in \mathbb{R}^n$ satisfy that $\left\| x_{P_i}-x_{E}\right\| > l_i$. Then, for any $\hat{x}\in \partial\mathcal{D}_i(x_{P_i},x_{E})$, for any  $\epsilon\in[0,\left\|\hat{x}-x_E \right\| )$,\\
	(a) $\hat{x}\in\partial\mathcal{D}_i\left(x_{P_i}+\epsilon\alpha_i\frac{\hat{x}-x_{P_i}}{\left\|\hat{x}-x_{P_i} \right\| },x_{E}+\epsilon\frac{\hat{x}-x_E}{\left\|\hat{x}-x_E \right\| }\right)$,\\
	(b) $\mathcal{D}_i \left(x_{P_i}+\epsilon\alpha_i\frac{\hat{x}-x_{P_i}}{\left\|\hat{x}-x_{P_i} \right\| },x_{E}+\epsilon\frac{\hat{x}-x_E}{\left\|\hat{x}-x_E \right\| }\right) \subseteq \mathcal{D}_i(x_{P_i},x_{E}) $.
\end{corollary}
\begin{proof}
	(a) By the conditions, we have
	\begin{align*}
		&-\left\| \hat{x}-\left( x_{P_i}+\epsilon\alpha_i\frac{\hat{x}-x_{P_i}}{\left\|\hat{x}-x_{P_i} \right\| }\right)\right\|  \\
		&+ \alpha_i\left\|  \hat{x}-\left( x_{E}+\epsilon\frac{\hat{x}-x_E}{\left\|\hat{x}-x_E \right\| }\right) \right\|+l_i \\
		=&-\left( \left\| \hat{x}- x_{P_i}\right\| - \epsilon\alpha_i\right) + \alpha_i \left(\left\| \hat{x}- x_E\right\| - \epsilon\right)  + l_i\\
		=&-\left\| \hat{x}- x_{P_i}\right\|+\alpha_i \left\| \hat{x}- x_E\right\| + l_i\\
		=&0.
	\end{align*}
	Thus,  $\hat{x}\in\partial\mathcal{D}_i\left(x_{P_i}+\epsilon\alpha_i\frac{\hat{x}-x_{P_i}}{\left\|\hat{x}-x_{P_i} \right\| },x_{E}+\epsilon\frac{\hat{x}-x_E}{\left\|\hat{x}-x_E \right\| }\right)$.\\
	(b) Let $x_{P_i}(t),x_E(t)$ satisfy
	\begin{align*}
		\dot{x}_{P_i}(t)&=\alpha_i\gamma_{P_i}\left( x_{P_i}(t),x_E(t),\frac{\hat{x}-x_E(t)}{\left\|\hat{x}-x_E(t))\right\| }\right) ,\\
		\dot{x}_{E}(t)&=\frac{\hat{x}-x_E(t)}{\left\|\hat{x}-x_E(t))\right\| },\\
		x_{P_i}(0)&=x_{P_i},\\
		x_E(0)&=x_E.
	\end{align*}
	It is easy to obtain that 
	\begin{align*}
		x_{P_i}(\epsilon)&=x_{P_i}+\epsilon\alpha_i\frac{\hat{x}-x_{P_i}}{\left\|\hat{x}-x_{P_i} \right\| }, ,\\
		x_E(\epsilon)&=x_{E}+\epsilon\frac{\hat{x}-x_E}{\left\|\hat{x}-x_E\right\| },
	\end{align*}
	when  $\epsilon\in[0,\left\|\hat{x}-x_E \right\| )$. From \autoref{th_pursuit_strategy}(b), we obtain the conclusion.
\end{proof}

\subsection{Differential properties of the marginal function in mathematical program}
In this part, we introduce some existing literature’s results about marginal functions and Clark’s generalized gradients. Consider a mathematical program problem involving parameters:
\begin{equation}\label{eq_prameteric_mp}
	\begin{split}
		\min_{x} &f_0(x,v),\ x\in\mathbb{R}^n,v\in\mathbb{R}^m\\
		\text{s.t. } & f_i(x,v)\le 0, i=1,\dots,s,
	\end{split}\tag{$\mathbf{P}_v$}
\end{equation}
where the vector $v$ is seen as a parameter vector. We assume that $f_i$ is locally Lipschitz continuous on $\mathbb{R}^n\times\mathbb{R}^m$. For each value of $v$, the set of feasible solution is
\begin{align*}
	S(v)\triangleq\left\lbrace x:f_i(x,v)\le 0, i=1,\dots,s\right\rbrace. 
\end{align*}
The marginal function is the optimal value of the program for each value of $v$:
\begin{align*}
	q(v)\triangleq\min \left\lbrace f_0(x,v):x\in S(v)\right\rbrace. 
\end{align*}
The set of optimal solution is
\begin{align*}
	P(v)\triangleq\left\lbrace x\in S(v): f_0(x,v)=q(v) \right\rbrace .
\end{align*}
Let $\lambda=(\lambda_1,\dots,\lambda_s)^\top$.
Let $K(x,v)$ be the set of $(\lambda,z)$ satisfying
\begin{equation}\label{eq_def_K}
	\begin{split}
		\lambda_i\ge0,\ i=1,\dots,s,\\
		\lambda_i f_i(x,v)=0,\ i=1,\dots,s,\\
		(0^\top,z^\top)^\top\in \partial \left[ f_0+\sum_{i=1}^s\lambda_i f_i \right](x,v). 
	\end{split}
\end{equation}
Let $K_0(x,v)$ be the set of $(\lambda,z)$ satisfying
\begin{equation}\label{eq_def_K0}
	\begin{split}
		\lambda_i\ge0,\ i=1,\dots,s,\\
		\lambda_i f_i(x,v)=0,\ i=1,\dots,s,\\
		(0^\top,z^\top)^\top \in \partial \left[ \sum_{i=1}^s\lambda_i f_i \right](x,v).
	\end{split}
\end{equation}


\begin{definition}[Tame\cite{rockafellar1982lagrange}]
	We say that for a given $v$ that program \eqref{eq_prameteric_mp} is tame if there exists $\delta_0>0$ and $\alpha_0>0$ with the property: there is a bounded mapping $\varphi$ from the set
	\begin{align*}
		\left\lbrace v':q(v')<\alpha_0 \text{ and } \left\| v'-v\right\| <\delta_0\right\rbrace 
	\end{align*}
	to $\mathbb{R}^n$ such that for every $v'$ in this set, $\varphi(v')$ is an optimal solution of ($\mathbf{P}_{v'}$).
\end{definition}

Below is Rockafellar's result \cite{rockafellar1982lagrange} regarding the differential properties of marginal functions. This theorem derives outer estimates for Clarke's generalized gradients in terms of Lagrange multiplier vectors that satisfy the Clarke's necessary conditions \cite{clarke1976new}.
\begin{theorem}[Corollary 2 in page 47 of \cite{rockafellar1982lagrange}]\label{th_rockafellar1982}
	Consider the program \eqref{eq_prameteric_mp} where $f_0,f_1\dots,f_s$ is locally Lipschitz continuous on $\mathbb{R}^n\times\mathbb{R}^m$. Suppose $v\in \mathbb{R}^m$ is such that \eqref{eq_prameteric_mp} is tame and $K_0(x,v)=\left\lbrace (0,0)\right\rbrace $ for all $x\in P(v)$. Then, $q(\cdot)$ is locally Lipschitz continuous near $v$ and
	\begin{align*}
		\partial q(v)\subseteq \overline{\text{co}\left\lbrace z:\exists \lambda \in\mathbb{R}^s, (\lambda,z)\in  \bigcup_{x\in P(v)}K(x,v)\right\rbrace }.
	\end{align*}
\end{theorem}

In the parametric mathematical program that we study (see \eqref{modified_program}), $f_0$ is locally Lipschitz continuous on $\mathbb{R}^n\times\mathbb{R}^m$ and $f_1,\dots,f_s$ are continuously differentiable on $\mathbb{R}^n\times\mathbb{R}^m$. This condition is between the conditions of \cite[Theorem 5.3]{gauvin1982differential} and \autoref{th_rockafellar1982}. \autoref{th_rockafellar1982} is too general for our study. We need a specialized proposition (\autoref{th_Rockafellar & Gauvin}) derived from \autoref{th_rockafellar1982}. Below are some necessary preliminaries of the specialized proposition.
\begin{definition}[Uniformly compact\cite{gauvin1982differential}]
	Let $E\subseteq\mathbb{R}^k$. A mapping $F:E\rightarrow 2^{\mathbb{R}^l}$ is said to be uniformly compact near $\bar{x}\in E$ if there is a neighborhood $N(\bar{x})$ of $\bar{x}$ such that the closure of $\bigcup_{x\in N(\bar{x})}F(x)$ is compact.
\end{definition}

\begin{proposition}\label{th_uniformly_compact->tame}
	Consider \eqref{eq_prameteric_mp} with a given parameter vector $v$. If $S(\cdot)$ is uniformly compact near $v$, then the program \eqref{eq_prameteric_mp} is tame at $v$.
\end{proposition}
\begin{proof}
	By the uniformly compactness, there exists $\delta_0>0$ such that $\bigcup_{v'\in B(v,\delta_0)}S(v')$ is bounded. Let $\phi$ be a mapping from $B(v,\delta_0)$ to $\mathbb{R}^n$ such that $\phi(v')\in P(v')$ for any $v'\in B(v,\delta_0)$. Since $P(v')\in S(v')\subseteq\bigcup_{v'\in B(v,\delta_0)}S(v')$ for any $v'\in B(v,\delta_0)$, $\phi$ is bounded. Thus, \eqref{eq_prameteric_mp} is tame at $v$.
\end{proof}

\begin{definition}[Mangasarian-Fromovitz regular]
	Consider the program \eqref{eq_prameteric_mp} where $f_1,\dots,f_s$ are continuously differentiable on $\mathbb{R}^n\times\mathbb{R}^m$. Given a parameter vector $v\in \mathbb{R}^m$, a feasible point $x\in S(v)$ is said to be Mangasarian-Fromovitz regular if there exists a vector $r\in\mathbb{R}^n$ such that
	\begin{equation}\label{eq_M-F_condition}
		\begin{split}
			&D_x f_i (x,v)^\top r<0,\\
			&\forall i\in I(x,v)\triangleq\left\lbrace i\in[s]:f_i(x,v)=0\right\rbrace.
		\end{split}
	\end{equation}
\end{definition}

\begin{proposition}\label{th_MF->K_0=0}
	Consider the program \eqref{eq_prameteric_mp} where $f_1,\dots,f_s$ are continuously differentiable on $\mathbb{R}^n\times\mathbb{R}^m$. Given a parameter vector $v\in \mathbb{R}^n$ and a feasible point $x\in S(v)$, if $x$ is Mangasarian-Fromovitz regular, then $K_0(x,v)=\left\lbrace (0,0)\right\rbrace $.
\end{proposition}
\begin{proof}
	It is easy to obtain that $(0,0)\in K_0(x,v)$. Select $(\lambda,z)$ from $K_0(x,v)$. According to \eqref{eq_def_K0},  \autoref{th_differentiable_Clarke} and \autoref{th_Clarke_linear_combination}, $(\lambda,z)$ satisfies
	\begin{subequations}
		\begin{align}
			\lambda_i\ge0,\ i=1,\dots,s,\label{lambda_nonnegtive_K0}\\
			\lambda_i f_i(x,v)=0,\ i=1,\dots,s,\label{complementary slackness}\\
			0=\sum_{i=1}^s\lambda_i D_x f_i(x,v),\label{nabla_x L = 0}\\
			z=\sum_{i=1}^s\lambda_i D_v f_i(x,v)\label{z_def_K0}.
		\end{align}
	\end{subequations}
	According to the Mangasarian-Fromovitz regularity condition, there exists a vector $r\in \mathbb{R}^n$ such that
	\begin{equation}\label{eq_M-f_in_K0}
		\begin{split}
			&D_x f_i (x,v)^\top r<0,\ \forall i\in I(x,v).
		\end{split}
	\end{equation}
	By \eqref{complementary slackness}, we have
	\begin{align}
		\lambda_i=0,\ \forall i\in[s]\setminus I(x,v).
	\end{align}
	Substituting the above equation into \eqref{nabla_x L = 0} and multiplying both sides of the equation by $r$, we have
	\begin{align}\label{sum_m-f}
		0=\sum_{i\in I(x,v)}\lambda_i D_x f_i(x,v)^\top r.
	\end{align}
	Combining \eqref{lambda_nonnegtive_K0}\eqref{eq_M-f_in_K0}\eqref{sum_m-f}, we have
	\begin{align*}
		\lambda_i=0,\ \forall i\in I(x,v).
	\end{align*}
	Then, $\lambda=0$. Substituting $\lambda=0$ into \eqref{z_def_K0}, we have $z=0$. Thus, $K_0(x,v)=\left\lbrace (0,0)\right\rbrace$.  
\end{proof}

\autoref{th_convergence_lambda} and \autoref{th_K_compactness} are variants of \cite[Theorem 3.4 and Corollary 3.6]{gauvin1982differential}. In \cite{gauvin1982differential}, $f_0$ and $f_1,\dots,f_s$ are all continuously differentiable.
\begin{proposition}\label{th_convergence_lambda}
	Consider the program \eqref{eq_prameteric_mp} where $f_0$ is locally Lipschitz continuous on $\mathbb{R}^n\times\mathbb{R}^m$ and $f_1,\dots,f_s$ are continuously differentiable on $\mathbb{R}^n\times\mathbb{R}^m$. Given a parameter  vector $v$, assume that there exists a optimal solution $x\in P(v)$ such that $x$ is Mangasarian-Fromovitz regular. Let $\left\lbrace x_n\right\rbrace $, $\left\lbrace v_n\right\rbrace$, $\left\lbrace\lambda^n \right\rbrace$, $\left\lbrace z_n \right\rbrace$ be sequences such that $x_n\in P(v_n)$, $(\lambda^n,z_n)\in K(x_n,v_n)$, $x_n\rightarrow x$, $v_n\rightarrow v$. Then, there exist subsequences $\left\lbrace x_m\right\rbrace $, $\left\lbrace v_m\right\rbrace$, $\left\lbrace\lambda^m \right\rbrace$, $\left\lbrace z_m \right\rbrace$ such that $(\lambda^m,z_m)\in K(x_m,v_m)$ and $\lambda^m\rightarrow \lambda$, $z_m\rightarrow z$ for some $(\lambda,z)\in K(x,v)$.
\end{proposition}

\begin{proof}
	 Let $r$ be given by Mangasarian-Fromovitz regularity condition, i.e. 
	\begin{align*}
		& D_xf_i(x,v)^\top r<0,\\
		&\forall i\in I (x,v)=\left\lbrace i\in[s]:f_i(x,v)=0\right\rbrace 
	\end{align*}
	By the continuity of function $D_xf_i$, for $n$ sufficiently large,
	\begin{equation}\label{eq_n_M-F}
		\begin{split}
			& D_xf_i(x_n,v_n)^\top r<0,\\
			&\forall i\in I (x,v)=\left\lbrace i\in[s]:f_i(x,v)=0\right\rbrace 
		\end{split}
	\end{equation}
	According to \autoref{th_Clarke_linear_combination} and \autoref{th_differentiable_Clarke}, for any $(x,v)$
	\begin{align*}
		\partial \left[ f_0+\sum_{i=1}^s\lambda_i f_i \right](x,v)=\partial f_0(x,v)+\sum_{i=1}^s\lambda_i Df_i (x,v)
	\end{align*}
	For each $(x_n,v_n)$, take $(\lambda^n,z_n)\in K(x_n,v_n)$. Then, there exists $ (p_n^\top,q_n^\top)^\top \in \partial f_0 (x_n,v_n)$ such that,
	\begin{subequations}\label{eq_n_K-T}
		\begin{align}
			0=p_n+\sum_{j=1}^s\lambda_j^n D_x f_j(x_n,v_n),\label{eq_n_K-T1}\\
			z_n=q_n+\sum_{j=1}^s\lambda_j^n D_v f_j(x_n,v_n).
		\end{align}
	\end{subequations}
	From \eqref{eq_n_M-F} and \eqref{eq_n_K-T1}, we have that
	\begin{align*}
		p_n^\top r=&-\sum_{j=1}^s\lambda_j^n D_x f_j(x_n,v_n)^\top r\\
		\ge &-\lambda_i^n D_x f_i(x_n,v_n)^\top r
	\end{align*}
	for any $i\in I (x,v)$. Then, for $n$ sufficiently large,
	\begin{equation}\label{eq_lambda_bounded}
		\begin{aligned}
			\lambda_i^n
			&\le \frac
			{-p_n^\top r}
			{D_xf_i(x_n,v_n)^\top r}\\
			&\le \frac
			{-L\left\| r\right\| }
			{D_xf_i(x_n,v_n)^\top r}  &\text{if } i\in I(x,v),\\
			\lambda_i^n&= 0 &\text{if } i\notin I(x,v),
		\end{aligned}
	\end{equation}
	where $L$ is a Lipschitz continuous rank of $f_0$ near $(x,v)$. From \eqref{eq_lambda_bounded}, the sequence $\left\lbrace \lambda^n\right\rbrace $ is bounded. In addition, the sequence $\left\lbrace (p_n^\top,q_n^\top)^\top\right\rbrace $ is also bounded according to \autoref{th_Clark_GG_bounded}. Therefore, there exist convergent subsequences $\left\lbrace \lambda^m\right\rbrace $ and $\left\lbrace (p_m^\top,q_m^\top)^\top\right\rbrace $.
	Let
	\begin{align*}
		\lambda&=\lim\limits_{m\rightarrow +\infty} \lambda^m,\\
		(p^\top,q^\top)^\top&=\lim\limits_{m\rightarrow +\infty}(p_m^\top,q_m^\top)^\top.
	\end{align*}
	Take $m\rightarrow +\infty$, by \eqref{eq_n_K-T}, we have
	\begin{align*}
		0=p+\sum_{j=1}^s\lambda_j D_x f_j(x,v),\\
		z=q+\sum_{j=1}^s\lambda_j D_v f_j(x,v).
	\end{align*}
	where $z=\lim\limits_{m\rightarrow+\infty}z_m$.
	By $ (p_n^\top,q_n^\top)^\top \in \partial f_0 (x_n,v_n)$ and the definition of Clarke's generalized gradient, we obtain $(p^\top,q^\top)^\top \in \partial f_0 (x,v)$. Thus, $(\lambda,z)\in K(x,v)$. The proof is completed.
\end{proof}

\begin{corollary}\label{th_K_compactness}
	Consider the program \eqref{eq_prameteric_mp} where $f_0$ is locally Lipschitz continuous on $\mathbb{R}^n\times\mathbb{R}^m$ and $f_1,\dots,f_s$ are continuously differentiable on $\mathbb{R}^n\times\mathbb{R}^m$. Given a parameter  vector $v$, assume that $S(v)$ is nonempty and compact and $x$ is Mangasarian-Fromovitz regular for any $x\in P(v)$. Then, $\bigcup_{x\in P(v)}K(x,v)$ is compact.
\end{corollary}

\begin{proof}
	Take a sequence $\left\lbrace (\lambda^n,z_n)\right\rbrace \subseteq \bigcup_{x\in P(v)}K(x,v) $. There exists a sequence $\left\lbrace x_n\right\rbrace$ such that $(\lambda^n,z_n)\in K(x_n,v)$. Since $P(v)$ is compact, there exist a subsequence $\left\lbrace x_m\right\rbrace$ and $x\in P(v)$ such that $x_m\rightarrow x$. Let $v_m=v$. Applying \autoref{th_convergence_lambda} to subsequences $\left\lbrace x_m\right\rbrace $, $\left\lbrace v_m\right\rbrace$, $\left\lbrace\lambda^m \right\rbrace$, $\left\lbrace z_m \right\rbrace$, there exists a subsequence $\left\lbrace (\lambda^k,z_k)\right\rbrace $ with $(\lambda^k,z_k)\in K(x_k,v)$ such that $(\lambda^k,z_k)\rightarrow (\lambda,z)\in K(x,v)$. Thus, $\bigcup_{x\in P(v)}K(x,v)$ is compact. The proof is completed.
\end{proof}

Combining \autoref{th_rockafellar1982}, \autoref{th_uniformly_compact->tame}, \autoref{th_MF->K_0=0} and \autoref{th_K_compactness}, we have the following corollary.
\begin{corollary}\label{th_Rockafellar & Gauvin}
	Consider the program \eqref{eq_prameteric_mp} where $f_0$ is locally Lipschitz continuous on $\mathbb{R}^n\times\mathbb{R}^m$ and $f_1,\dots,f_s$ are continuously differentiable on $\mathbb{R}^n\times\mathbb{R}^m$. Suppose $v\in \mathbb{R}^m$ is such that $S(\cdot)$ is uniformly compact near $v$ and $x$ is Mangasarian-Fromovitz regular for all $x\in P(v)$. Then, $q(\cdot)$ is locally Lipschitz continuous near $v$ and
	\begin{align*}
		\partial q(v)\subseteq {\text{co}\left\lbrace z:\exists \lambda \in\mathbb{R}^s, (\lambda,z)\in  \bigcup_{x\in P(v)}K(x,v)\right\rbrace }.
	\end{align*}
\end{corollary}
\begin{proof}
	By \autoref{th_uniformly_compact->tame} and \autoref{th_MF->K_0=0}, the tameness and $K_0(x,v)=\left\lbrace (0,0)\right\rbrace $ is satisfied. The conclusion as in \autoref{th_rockafellar1982} is derived. By \autoref{th_K_compactness}, $\left\lbrace z:\exists \lambda \in\mathbb{R}^s, (\lambda,z)\in  \bigcup_{x\in P(v)}K(x,v)\right\rbrace $ is compact. The closure symbol as in \autoref{th_rockafellar1982} can be removed.
\end{proof}

\section{Verification of viscosity solution}\label{sec_verification of viscosity solution}
In this section, we prove that $V^g$ is the viscosity solution of the HJI PDE in \eqref{HJI}.  First, we estimate the Clark's generalized gradient of $V^g$ in Subsection \ref{subsec_estimate}. By this estimate, we show that $V^g$ is the viscosity subsolution and supersolution of the HJI PDE in Subsection \ref{subsec_subsolution} and \ref{subsec_supersolution}, respectively. We summarize the main results in Subsection \ref{subsec_Summary}.
\subsection{Estimate of the Clark's generalized gradient of $V^g$}\label{subsec_estimate}
The method used in \cite{lee2024solutions} to obtain the Clarke's generalized gradient is not applicable when $g$ is not convex. In this study, we use Rockafellar's results (\autoref{th_Rockafellar & Gauvin}) to estimate the Clark's generalized gradient of $V^g$.

According to the definition, $V^g(\mathbf{y})$ is the optimal value of the following mathematical program for each value of $\mathbf{y}=(x_{P_1}^\top,\dots,x_{P_m}^\top,x_E^\top)^\top$,
\begin{equation}\label{program}
	\begin{split}
		&\min_{x\in \mathbb{R}^n}g(x)\\
		\text{s.t. } &d_i(x;x_{P_i},x_E) \le 0, \\
		&i=1,\dots,m.
	\end{split}
\end{equation}
As $\left\| \cdot \right\| $ is not differentiable at the origin, the program needs to be modified for convenience. Let
\begin{align*}
	\hat{d}_i(x;x_{P_i},x_E)\triangleq&-((x-x_{P_i})^2-\alpha_i^2(x-x_E)^2-l_i^2)^2\\
	&+4\alpha_i^2 l_i^2 (x-x_E)^2,\\
	\bar{d}_i(x;x_{P_i},x_E)\triangleq&-((x-x_{P_i})^2-\alpha_i^2(x-x_E)^2-l_i^2).
\end{align*}
It is obvious that $\hat{d}_i$ and $\bar{d}_i$ are smooth in $\mathbb{R}^n\times\mathbb{R}^n\times\mathbb{R}^n$.
\begin{proposition}\label{constraint_equivalent_2}
	Assume that  $l_i>0$ and $\left\| x_{P_i}-x_E\right\| >l_i$. The following statements are equivalent:\\
	(a) $d_i(x;x_{P_i},x_E) \le 0$;\\
	(b) $\hat{d}_i(x;x_{P_i},x_E)\le 0$ and $\bar{d}_i(x;x_{P_i},x_E)<0$;\\
	(c) $\hat{d}_i(x;x_{P_i},x_E)\le 0$ and $\bar{d}_i(x;x_{P_i},x_E)\le0$.\\
	The following statements are also equivalent:\\
	(d) $d_i(x;x_{P_i},x_E) = 0$;\\
	(e) $\hat{d}_i(x;x_{P_i},x_E)= 0$ and $\bar{d}_i(x;x_{P_i},x_E)<0$;\\
	(f) $\hat{d}_i(x;x_{P_i},x_E)= 0$ and $\bar{d}_i(x;x_{P_i},x_E)\le0$.
\end{proposition}
\begin{proof}
	We only prove that (a), (b) and (c) are equivalent. It is easy to obtain 
	\begin{equation}\label{eq_constraint_equivalent}
		\begin{split}
			&-\left\| x-x_{P_i}\right\|+\alpha_i \left\| x-x_E \right\| + l_i \le 0\\
			\iff& \alpha_i \left\| x-x_E \right\| + l_i \le \left\| x-x_{P_i}\right\| \\
			\iff& \alpha_i^2(x-x_E)^2+l_i^2 + 2 \alpha_i l_i \left\| x-x_E \right\| \le (x-x_{P_i})^2 \\
			\iff&-((x-x_{P_i})^2-\alpha_i^2(x-x_E)^2-l_i^2)\\
				&\le -2 \alpha_i l_i \left\| x-x_E \right\|.
		\end{split}
	\end{equation}
	Let us prove that (a) implies (b). $x=x_E$ does not satisfy (a). Thus, from the last inequality in \eqref{eq_constraint_equivalent}, (a) implies $-((x-x_{P_i})^2-\alpha_i^2(x-x_E)^2-l_i^2)<0$. Squaring on both sides of the last inequality of \eqref{eq_constraint_equivalent} and reorganizing it, we obtain $-((x-x_{P_i})^2-\alpha_i^2(x-x_E)^2-l_i^2)^2+4\alpha_i^2 l_i^2 (x-x_E)^2\le 0$.\\
	Obviously, (b) implies (c).\\
	Let us prove that (c) implies (a).
	\begin{align*}
		&\hat{d}_i(x;x_{P_i},x_E)\le 0\\
		\iff& 4\alpha_i^2 l_i^2 (x-x_E)^2\le ((x-x_{P_i})^2-\alpha_i^2(x-x_E)^2-l_i^2)^2 \\
		\iff& 2 \alpha_i l_i \left\| x-x_E \right\| \le \left| (x-x_{P_i})^2-\alpha_i^2(x-x_E)^2-l_i^2 \right| . 
	\end{align*}
	Due to $-((x-x_{P_i})^2-\alpha_i^2(x-x_E)^2-l_i^2)\le 0$, we obtain $2 \alpha_i l_i \left\| x-x_E \right\| \le  (x-x_{P_i})^2-\alpha_i^2(x-x_E)^2-l_i^2 $. From \eqref{eq_constraint_equivalent}, we obtain (a). \\
	The proof of the equivalence of (d), (e) and (f) is similar.
\end{proof}

From \autoref{constraint_equivalent_2}, the mathematical program \eqref{program} is equivalent to the following program,
\begin{equation}\label{modified_program}
	\begin{split}
		&\min_{x\in \mathbb{R}^n}g(x)\\
		\text{s.t. } &\hat{d}_i(x;x_{P_i},x_E)\le 0,\; i=1,\dots,m,\\
		&\bar{d}_i(x;x_{P_i},x_E)\le 0,\;i=1,\dots,m.
	\end{split}
\end{equation}
$V^g(\mathbf{y})$ is the optimal value of the program for each value of $\mathbf{y}=(x_{P_1}^\top,\dots,x_{P_m}^\top,x_E^\top)^\top$, which can be regarded as a marginal function\cite{gauvin1982differential}. The set of feasible solutions is $\mathcal{D}^*(\mathbf{y})$. 
The set  of optimal solutions is
\begin{align*}
	\mathcal{P}(\mathbf{y})=\left\lbrace x\in\mathcal{D}^*(\mathbf{y}):g(x)=V^g(\mathbf{y})\right\rbrace. 
\end{align*}
For any optimal point $x\in\mathcal{P}(\mathbf{y})$, let $\mathcal{K}(x,\mathbf{y})$ be the set of $(\lambda_1,\dots,\lambda_m,\mu_1,\dots,\mu_m)\in \mathbb{R}^{2m}$ such that
\begin{subequations}\label{K-T}
	\begin{align}
		\lambda_i \ge 0,\ \mu_i \ge 0\ i=1,\dots,m,\label{lagrange_nonnegtive_1}\\
		\lambda_i \hat{d}_i (x;x_{P_i},x_E)=0,\ i=1,\dots,m,\\
		\mu_i \bar{d}_i (x;x_{P_i},x_E)=0,\ i=1,\dots,m,\label{complementary slackness mu 1}\\
		\begin{matrix}
			0\in \partial_x g(x)+\sum_{i=1}^m \lambda_i D_x \hat{d}_i (x;x_{P_i},x_E)\\[6pt]
			+\sum_{i=1}^m \mu_i D_x \bar{d}_i (x;x_{P_i},x_E).
		\end{matrix} 
	\end{align}
\end{subequations}
By \eqref{lagrange_nonnegtive_1}\eqref{complementary slackness mu 1} and ``(b)$\iff$(c)'' in \autoref{constraint_equivalent_2}, $\forall (\lambda_1,\dots,\lambda_m,\mu_1,\dots,\mu_m)\in\mathcal{K}(x,\mathbf{y})$, $\forall i\in[m]$, $\mu_i=0$.

\begin{proposition}\label{th_M-F}
	Consider the program \eqref{modified_program}. Assume that $\alpha_i>1$, $l_i> 0, i=1,\dots,m$. Then, for any $\mathbf{y}\in\Omega$, for any $x\in \mathcal{P}(\mathbf{y})$, $x$ is Mangasarian-Fromovitz regular.
\end{proposition} 
The proof is presented in the supplementary material.

\begin{proposition}\label{th_uniform_compactness}
	Assume that $\alpha_i>1$, $l_i\ge 0, i=1,\dots,m$. Then, the point to set mapping $\mathcal{D}^*(\mathbf{y})$ is uniformly compact near any point of $\Omega$.
\end{proposition}
The proof is presented in the supplementary material.

Below is the estimate of the Clark's generalized gradient of $V^g$ in $\Omega$.
\begin{proposition}\label{estimate of generalized gradient}
	Assume that $\alpha_i>1$, $l_i>0$ for any $i\in[m]$ and $g:\mathbb{R}^n\rightarrow \mathbb{R}$ is locally Lipschitz continuous on $\mathbb{R}^n$. Then,
	$V^g$ is locally Lipschitz continuous on $\Omega$, and for any $\mathbf{y}=(x_{P_1}^\top,\dots,x_{P_m}^\top,x_E^\top)^\top\in\Omega$,
	\begin{align}\label{eq estimate of generalized gradient}
		\partial V^g (\mathbf{y})\subseteq \mathrm{co} \hat{\mathcal{Q}}(\mathbf{y}),
	\end{align}
	where $\hat{\mathcal{Q}}(\mathbf{y})$ is the set of  $\mathbf{p}=(p_{P_1}^\top,\dots,p_{P_m}^\top,p_E^\top)^\top$ such that $\exists x\in \mathcal{P}(\mathbf{y})$, $\exists (\lambda_1,\dots,\lambda_m)\in \mathbb{R}^m$,
	\begin{align*}
		\lambda_i \ge 0,\ i=1,\dots,m,\\
		\lambda_i \hat{d}_i (x;x_{P_i},x_E)=0,\ i=1,\dots,m,\\
		0\in \partial_x g(x)+\sum_{i=1}^m \lambda_i D_x \hat{d}_i (x;x_{P_i},x_E) 
	\end{align*}
	(i.e. $(\lambda_1,\dots,\lambda_m,0,\dots,0)\in\mathcal{K}(x,\mathbf{y})$) and
	\begin{align*}
		&p_{P_i}=8\alpha_i l_i \lambda_i \left\| x-x_E\right\| (x-x_{P_i}), i=1,\dots,m,\\
		&p_{E}=-\sum_{i=1}^m 8\alpha_i^2 l_i \lambda_i \left\| x-x_{P_i}\right\| (x-x_E).
	\end{align*}
\end{proposition}
\begin{proof} 
	The Mangasarian-Fromovitz regularity condition and uniform compactness have been shown in \autoref{th_M-F} and \autoref{th_uniform_compactness} respectively.
	According to \autoref{th_Rockafellar & Gauvin}, we obtain that $V^g$ is locally Lipschitz continuous in $\Omega$ and for any $\mathbf{y}\in \Omega$,
	\begin{align}
		\partial V^g (\mathbf{y})\subseteq \mathrm{co} \mathcal{Q}(\mathbf{y}),
	\end{align}
	where $\mathcal{Q}(\mathbf{y})$ is the set of $\mathbf{p}=(p_{P_1}^\top,\dots,p_{P_m}^\top,p_E^\top)^\top$ such that $\exists x\in \mathcal{P}(\mathbf{y})$, $\exists (\lambda_1,\dots,\lambda_m,\mu_1,\dots\mu_m)\in \mathbb{R}^{2m}$
	\begin{subequations}
		\begin{align}
			\lambda_i \ge 0,\ \mu_i \ge 0,\ i=1,\dots,m,\label{lagrange_nonnegtive}\\
			\lambda_i \hat{d}_i (x;x_{P_i},x_E)=0,\ i=1,\dots,m,\\
			\mu_i \bar{d}_i (x;x_{P_i},x_E)=0,\ i=1,\dots,m,\label{complementary slackness mu}\\
			(0^\top,\mathbf{p}^\top)^\top
			\in \partial_{(x,\mathbf{y})} \left[ g+\sum_{i=1}^m \left( \lambda_i\hat{d}_i+\mu_i\bar{d}_i\right) \right] (x,\mathbf{y}).\label{eq}
		\end{align}
	\end{subequations}
	By \eqref{lagrange_nonnegtive}\eqref{complementary slackness mu} and ``(b)$\iff$(c)'' in \autoref{constraint_equivalent_2}, $\mu_i=0,\ \forall i\in[m]$.
	According to \autoref{th_Clarke_linear_combination},
	\begin{align*}
		&\partial_{(x,\mathbf{y})} \left[ g+\sum_{i=1}^m \left( \lambda_i\hat{d}_i+\mu_i\bar{d}_i\right) \right] (x,\mathbf{y})\\
		=&\partial_{(x,\mathbf{y})} g(x,\mathbf{y})+\sum_{i=1}^m D_{(x,\mathbf{y})} \left( \lambda_i\hat{d}_i+\mu_i\bar{d}_i\right) (x,\mathbf{y}).
	\end{align*}
	Since the function $g$ does not depend on the variable $\mathbf{y}$,
	\begin{align*}
		\partial_{(x,\mathbf{y})} g(x,\mathbf{y})=\partial_x g(x)\times\left\lbrace 0 \right\rbrace .
	\end{align*}
	Thus, \eqref{eq} is converted to 
	\begin{align*}
		0\in \partial_x g(x)+\sum_{i=1}^m \lambda_i D_x \hat{d}_i (x;x_{P_i},x_E)
	\end{align*}
	and
	\begin{align*}
		&p_{P_i}=\lambda_i D_{x_{P_i}}\hat{d}_i(x;x_{P_i},x_E), i=1,\dots,m,\\
		&p_{E}=\sum_{i=1}^m\lambda_i D_{x_E}\hat{d}_i(x;x_{P_i},x_E).
	\end{align*}
	Consider $\lambda_i D_{x_{P_i}}\hat{d}_i(x;x_{P_i},x_E)$ and $\lambda_i D_{x_E}\hat{d}_i(x;x_{P_i},x_E)$,
	\begin{equation}\label{sssss}
		\begin{aligned}
			&\lambda_i D_{x_{P_i}}\hat{d}_i(x;x_{P_i},x_E)\\
			=&4\lambda_i((x-x_{P_i})^2-\alpha_i^2(x-x_E)^2-l_i^2)(x-x_{P_i}),\\
			&\lambda_i D_{x_E}\hat{d}_i(x;x_{P_i},x_E)\\
			=&-4\alpha_i^2 \lambda_i((x-x_{P_i})^2-\alpha_i^2(x-x_E)^2-l_i^2+2l_i^2)(x-x_E).
		\end{aligned}
	\end{equation}
	When $\lambda_i=0$, 
	\begin{align*}
		&\lambda_i D_{x_{P_i}}\hat{d}_i(x;x_{P_i},x_E)=0=8\alpha_i l_i \lambda_i \left\| x-x_E\right\| (x-x_{P_i}),\\
		&\lambda_i D_{x_{E}}\hat{d}_i(x;x_{P_i},x_E)=0=-8\alpha_i^2 l_i \lambda_i \left\| x-x_{P_i}\right\| (x-x_E).
	\end{align*}
	When $\lambda_i>0$, by \eqref{K-T}, we obtain $\hat{d}_i(x;x_{P_i},x_E)=0$, then
	\begin{align*}
		&-((x-x_{P_i})^2-\alpha_i^2(x-x_E)^2-l_i^2)^2+4\alpha_i^2 l_i^2 (x-x_E)^2=0,\\
		&\left\| x-x_{P_i}\right\| -\alpha_i\left\| x-x_E\right\| -l_i=0.
	\end{align*}
	Substituting them into \eqref{sssss}, we obatin
	\begin{align*}
		&\lambda_i D_{x_{P_i}}\hat{d}_i(x;x_{P_i},x_E)=8\alpha_i l_i \lambda_i \left\| x-x_E\right\| (x-x_{P_i}),\\
		&\lambda_i D_{x_{E}}\hat{d}_i(x;x_{P_i},x_E)=-8\alpha_i^2 l_i \lambda_i \left\| x-x_{P_i}\right\| (x-x_E).
	\end{align*}
	Then $\hat{\mathcal{Q}}(\mathbf{y})=\mathcal{Q}(\mathbf{y})$. The proof is completed.
\end{proof}

\subsection{Verification of viscosity subsolution}\label{subsec_subsolution}
\begin{proposition}\label{th sub}
	Assume that $\alpha_i>1$, $l_i>0$ for any $i\in[m]$ and $g:\mathbb{R}^n\rightarrow \mathbb{R}$ is locally Lipschitz continuous on $\mathbb{R}^n$. Then $\forall\mathbf{y}\in\Omega$, $\forall \mathbf{p}=(p_{P_1}^\top,\dots,p_{P_m}^\top,p_E^\top)^\top \in \partial V^g(\mathbf{y}),$
	\begin{align}\label{eq sub}
		-\sum_{i=1}^m\alpha_i\left\| {p}_{P_i}\right\| +\left\| {p}_E\right\| \le 0.
	\end{align}
\end{proposition}
\begin{proof}
	\begin{figure*}[t]
		\begin{align}\label{p_{P_i}^2}
			p_{P_i}^2=
			\sum_{k=1}^{n(m+1)+1}\sum_{l=1}^{n(m+1)+1}
			8^2  \alpha_i^2 l_i^2 t_k t_l \lambda_i^{(k)}\lambda_i^{(l)}
			\left\| \bar{x}_k-x_E\right\| \left\| \bar{x}_l-x_E\right\|
			(\bar{x}_k-x_{P_i})^\top(\bar{x}_l-x_{P_i}).
		\end{align}
		\begin{equation}\label{p_{P_i}^2_is_greater}
			\begin{aligned}
				p_{P_i}^2&\ge
				\sum_{k=1}^{n(m+1)+1}\sum_{l=1}^{n(m+1)+1}
				8^2  \alpha_i^2 l_i^2 t_k t_l \lambda_i^{(k)}\lambda_i^{(l)}
				\left\| \bar{x}_k-x_{P_i}\right\| \left\| \bar{x}_l-x_{P_i}\right\|
				(\bar{x}_k-x_E)^\top(\bar{x}_l-x_E)\\
				&=\left( \sum_{k=1}^{n(m+1)+1} t_k8\alpha_i l_i \lambda_i^{(k)} \left\| \bar{x}_k-x_{P_i}\right\|(\bar{x}_k-x_E)\right)^2 =\left( p_E^{(i)}\right) ^2.
			\end{aligned}
		\end{equation}
	\end{figure*}
	Pick $\mathbf{p}=(p_{P_1}^\top,\dots,p_{P_m}^\top,p_E^\top)^\top \in \partial V^g(\mathbf{y})$. By \autoref{estimate of generalized gradient} and Carathéodory theorem \cite{rockafellar1997convex}, $\mathbf{p}$ can be represented as follows:
	\begin{align*}
		p_{P_i}=&\sum_{k=1}^{n(m+1)+1} t_k8\alpha_i l_i \lambda_i^{(k)} \left\| \bar{x}_k-x_E\right\|(\bar{x}_k-x_{P_i}), \\
		&i=1,\dots,m,\\
		p_{E}=&-\sum_{k=1}^{n(m+1)+1} t_k\sum_{i=1}^m8\alpha_i^2 l_i \lambda_i^{(k)} \left\| \bar{x}_k-x_{P_i}\right\|(\bar{x}_k-x_E),
	\end{align*}
	where 
	\begin{align*}
		&\bar{x}_k\in\mathcal{P}(\mathbf{y}),\\
		&(\lambda_1^{(k)},\dots,\lambda_m^{(k)},0,\dots,0)\in\mathcal{K}(\bar{x}_k,\mathbf{y}),\\
		&t_k \ge 0, k=1,\dots,n(m+1)+1,\\ 
		&\sum_{k=1}^{n(m+1)+1}t_k=1 .
	\end{align*}
	For convenience, let
	\begin{align*}
		p_E^{(i)}=\sum_{k=1}^{n(m+1)+1} t_k8\alpha_i l_i \lambda_i^{(k)} \left\| \bar{x}_k-x_{P_i}\right\|(\bar{x}_k-x_E).
	\end{align*}
	Calculate $p_{P_i}^2$, see \eqref{p_{P_i}^2} at the top of the page.\\
	If $\lambda_i^{(k)}\lambda_i^{(l)}=0$, then
	\begin{align*}
		&\lambda_i^{(k)}\lambda_i^{(l)}
		\left\| \bar{x}_k-x_E\right\| \left\| \bar{x}_l-x_E\right\|
		(\bar{x}_k-x_{P_i})^\top(\bar{x}_l-x_{P_i})\\
		=&
		\lambda_i^{(k)}\lambda_i^{(l)}
		\left\| \bar{x}_k-x_{P_i}\right\| \left\| \bar{x}_l-x_{P_i}\right\|
		(\bar{x}_k-x_E)^\top(\bar{x}_l-x_E).
	\end{align*}
	If $\lambda_i^{(k)}\lambda_i^{(l)}>0$, by \eqref{K-T}, we obtain that $\bar{x}_k,\bar{x}_l\in\partial\mathcal{D}_i(x_{P_i},x_E)$. Then, by \autoref{key_C-oval_property}, we have that
	\begin{align*}
		&\lambda_i^{(k)}\lambda_i^{(l)}
		\left\| \bar{x}_k-x_E\right\| \left\| \bar{x}_l-x_E\right\|
		(\bar{x}_k-x_{P_i})^\top(\bar{x}_l-x_{P_i})\\
		\ge&
		\lambda_i^{(k)}\lambda_i^{(l)}
		\left\| \bar{x}_k-x_{P_i}\right\| \left\| \bar{x}_l-x_{P_i}\right\|
		(\bar{x}_k-x_E)^\top(\bar{x}_l-x_E).
	\end{align*}
	Thus, we obtain \eqref{p_{P_i}^2_is_greater} which is shown at the top of the page. This implies $\left\| p_{P_i}\right\| \ge  \left\| p_E^{(i)}\right\|$. Then,
	\begin{align*}
		\sum_{i=1}^m\alpha_i\left\| p_{P_i}\right\| 
		\ge\sum_{i=1}^m\alpha_i\left\| p_E^{(i)}\right\| 
		\ge\left\| \sum_{i=1}^m\alpha_ip_E^{(i)}\right\|  
		=\left\| p_E\right\| .
	\end{align*}
	The proof is completed.
\end{proof}
\begin{corollary}\label{th_sub2}
	Assume that $\alpha_i>1$, $l_i>0$ for any $i\in[m]$ and $g:\mathbb{R}^n\rightarrow \mathbb{R}$ is locally Lipschitz continuous on $\mathbb{R}^n$. Then,
	$\forall\mathbf{y}\in\Omega$, $\forall\mathbf{p}=(p_{P_1}^\top,\dots,p_{P_m}^\top,p_E^\top)^\top \in D^+ V^g(\mathbf{y}),$
	\begin{align*}
		-\sum_{i=1}^m\alpha_i\left\| {p}_{P_i}\right\| +\left\| {p}_E\right\| \le 0.
	\end{align*}
\end{corollary}
\begin{proof}
	This follows from \autoref{th sub} and \autoref{th_D+D-partial}.
\end{proof}
\subsection{Verification of viscosity supersolution}\label{subsec_supersolution}
\begin{proposition}\label{th_super}
	Assume that $\alpha_i>1$, $l_i>0$ for any $i\in[m]$ and $g:\mathbb{R}^n\rightarrow \mathbb{R}$ is locally Lipschitz continuous on $\mathbb{R}^n$. Then, 
	$\forall\mathbf{y}\in\Omega$, $\forall\mathbf{p}=(p_{P_1}^\top,\dots,p_{P_m}^\top,p_E^\top)^\top \in D^- V^g(\mathbf{y}),$
	\begin{align}\label{eq_super}
		-\sum_{i=1}^m\alpha_i\left\| {p}_{P_i}\right\| +\left\| {p}_E\right\| = 0.
	\end{align}
\end{proposition}
\begin{proof}
	\begin{figure*}
		\begin{equation}\label{pz}
			\begin{aligned}
				&\mathbf{p}^\top \mathbf{z}
				=&\sum_{i=1}^m \sum_{k=1}^{n(m+1)+1} 8\alpha_i^2l_i t_k \lambda_i^{(k)}
				\left( 
				\left\| \bar{x}_k-x_E \right\|
				\left( \bar{x}_k-x_{P_i} \right) ^\top 
				\frac{\left( \hat{x}_{k_1}^{(i)}-x_{P_i}\right)}{\left\| \hat{x}_{k_1}^{(i)}-x_{P_i} \right\|} 
				-
				\left\| \bar{x}_k-x_{P_i} \right\|
				\left( \bar{x}_k-x_E \right) ^\top 
				\frac{\left( \bar{x}_{k_1}-x_E\right) }{\left\| \bar{x}_{k_1}-x_E \right\|} 
				\right) .
			\end{aligned}
		\end{equation}
		\begin{equation}\label{=0}
			\begin{aligned}
				8\alpha_i^2l_i t_k \lambda_i^{(k)}
				\left( 
				\left\| \bar{x}_k-x_E \right\|
				\left( \bar{x}_k-x_{P_i} \right) ^\top 
				\frac{\left( \hat{x}_{k_1}^{(i)}-x_{P_i}\right)}{\left\| \hat{x}_{k_1}^{(i)}-x_{P_i} \right\|} 
				-
				\left\| \bar{x}_k-x_{P_i} \right\|
				\left( \bar{x}_k-x_E \right) ^\top 
				\frac{\left( \bar{x}_{k_1}-x_E\right) }{\left\| \bar{x}_{k_1}-x_E \right\|} 
				\right) 
				=0.
			\end{aligned}
		\end{equation}
		\begin{equation}\label{>=0}
			\begin{aligned}
					&	8\alpha_i^2l_i t_k \lambda_i^{(k)}
				\left( 
				\left\| \bar{x}_k-x_E \right\|
				\left( \bar{x}_k-x_{P_i} \right) ^\top 
				\frac{\left( \hat{x}_{k_1}^{(i)}-x_{P_i}\right)}{\left\| \hat{x}_{k_1}^{(i)}-x_{P_i} \right\|} 
				-
				\left\| \bar{x}_k-x_{P_i} \right\|
				\left( \bar{x}_k-x_E \right) ^\top 
				\frac{\left( \bar{x}_{k_1}-x_E\right) }{\left\| \bar{x}_{k_1}-x_E \right\|} 
				\right)  \\
				=&8\alpha_i^2 l_i t_{k_1} \lambda_i^{(k)} \left\| \bar{x}_k-x_E \right\|	\left\| \bar{x}_k-x_{P_i} \right\|
				\left( 
				\frac{\left( \bar{x}_k-x_{P_i} \right) ^\top\left( \hat{x}_{k_1}^{(i)}-x_{P_i}\right)}{\left\| \bar{x}_k-x_{P_i} \right\| \left\| \hat{x}_{k_1}^{(i)}-x_{P_i} \right\| }
				-
				\frac{\left( \bar{x}_k-x_E \right) ^\top\left( \hat{x}_{k_1}^{(i)}-x_E\right)}{\left\| \bar{x}_k-x_E \right\| \left\| \hat{x}_{k_1}^{(i)}-x_E \right\| } 
				\right) \\
				\ge&0.
			\end{aligned}
		\end{equation}
		\begin{equation}\label{>0}
			\begin{aligned}
				8\alpha_{i_2}^2 l_{i_2} t_{k_2} \lambda_{i_2}^{(k_2)} \left\| \bar{x}_{k_2}-x_E \right\|	\left\| \bar{x}_{k_2}-x_{P_{i_2}} \right\|
				\left( 
				\frac{\left( \bar{x}_{k_2}-x_{P_{i_2}} \right) ^\top\left( \hat{x}_{k_1}^{(i_2)}-x_{P_{i_2}}\right)}{\left\| \bar{x}_{k_2}-x_{P_{i_2}} \right\| \left\| \hat{x}_{k_1}^{(i_2)}-x_{P_{i_2}} \right\| }
				-
				\frac{\left( \bar{x}_{k_2}-x_E \right) ^\top\left( \hat{x}_{k_1}^{(i_2)}-x_E\right)}{\left\| \bar{x}_{k_2}-x_E \right\| \left\| \hat{x}_{k_1}^{(i_2)}-x_E \right\| } 
				\right) 
				>0.
			\end{aligned}
		\end{equation}
		\begin{equation}\label{v(x+ez)=v(x)--1}
			\begin{aligned}
				&\rho_i\left( x_{P_i}+\epsilon \alpha_i \frac{\hat{x}_{k_1}^{(i)}-x_{P_i}}{\left\| \hat{x}_{k_1}^{(i)}-x_{P_i}\right\| },x_E+\epsilon \frac{\bar{x}_{k_1}-x_E}{\left\| \bar{x}_{k_1}-x_E\right\| },\frac{\bar{x}_{k_1}-x_E}{\left\| \bar{x}_{k_1}-x_E\right\| }\right) \\
				=&\left\|
				\hat{x}_{k_1}^{(i)}-
				\left( x_E+\epsilon \frac{\bar{x}_{k_1}-x_E}{\left\| \bar{x}_{k_1}-x_E\right\| }\right) \right\|
				=\left\| \hat{x}_{k_1}^{(i)}-x_E\right\| -\epsilon
				=\rho_i\left( x_{P_i},x_E,\frac{\bar{x}_{k_1}-x_E}{\left\| \bar{x}_{k_1}-x_E\right\| }\right) -\epsilon.
			\end{aligned}
		\end{equation}
		\begin{equation}\label{v(x+ez)=v(x)--3}
			\begin{aligned}
				\bar{x}_{k_1}
				&=x_E+ \min_i\rho_i\left( x_{P_i},x_E,\frac{\bar{x}_{k_1}-x_E}{\left\| \bar{x}_{k_1}-x_E\right\| }\right)\frac{\bar{x}_{k_1}-x_E}{\left\| \bar{x}_{k_1}-x_E\right\| }\\
				&=x_E+ \min_i\left( \rho_i\left( x_{P_i},x_E,\frac{\bar{x}_{k_1}-x_E}{\left\| \bar{x}_{k_1}-x_E\right\| }\right)-\epsilon\right) \frac{\bar{x}_{k_1}-x_E}{\left\| \bar{x}_{k_1}-x_E\right\| }+\epsilon \frac{\bar{x}_{k_1}-x_E}{\left\| \bar{x}_{k_1}-x_E\right\| }\\	
				&=\left( x_E+\epsilon \frac{\bar{x}_{k_1}-x_E}{\left\| \bar{x}_{k_1}-x_E\right\| }\right) + 
				\min_i \rho_i\left( x_{P_i}+\epsilon \alpha_i \frac{\hat{x}_{k_1}^{(i)}-x_{P_i}}{\left\| \hat{x}_{k_1}^{(i)}-x_{P_i}\right\| },x_E+\epsilon \frac{\bar{x}_{k_1}-x_E}{\left\| \bar{x}_{k_1}-x_E\right\| },\frac{\bar{x}_{k_1}-x_E}{\left\| \bar{x}_{k_1}-x_E\right\| }\right)\frac{\bar{x}_{k_1}-x_E}{\left\| \bar{x}_{k_1}-x_E\right\| }.
			\end{aligned}
		\end{equation}
	\end{figure*}
	Pick $\mathbf{y}=(x_{P_1}^\top,\dots,x_{P_m}^\top,x_E^\top)^\top\in\Omega$. Pick $\mathbf{p}=(p_{P_1}^\top,\dots,p_{P_m}^\top,p_E^\top)^\top \in D^- V^g(\mathbf{y})$. By \autoref{th_D+D-partial}, \autoref{estimate of generalized gradient} and Carathéodory theorem \cite{rockafellar1997convex}, $\mathbf{p}$ can be represented as follows:
	\begin{align*}
		p_{P_i}=&\sum_{k=1}^{n(m+1)+1} t_k8\alpha_i l_i \lambda_i^{(k)} \left\| \bar{x}_k-x_E\right\|(\bar{x}_k-x_{P_i}), \\
		&i=1,\dots,m,\\
		p_{E}=&-\sum_{k=1}^{n(m+1)+1} t_k\sum_{i=1}^m8\alpha_i^2 l_i \lambda_i^{(k)} \left\| \bar{x}_k-x_{P_i}\right\|(\bar{x}_k-x_E),
	\end{align*}
	where 
	\begin{align*}
		&\bar{x}_k\in\mathcal{P}(\mathbf{y}),\\
		&(\lambda_1^{(k)},\dots,\lambda_m^{(k)},0,\dots,0)\in\mathcal{K}(\bar{x}_k,\mathbf{y}),\\
		&t_k \ge 0, k=1,\dots,n(m+1)+1,\\ 
		&\sum_{k=1}^{n(m+1)+1}t_k=1 .
	\end{align*}
		Next, we show that $\left\lbrace \bar{x}_k:\exists i,t_k\lambda_i^{(k)}>0\right\rbrace$ is either a singleton or an empty set.
	For the sake of contradiction, assume that there exist $ k_1,k_2,i_1,i_2$ such that $\bar{x}_{k_1}\neq\bar{x}_{k_2} $, $t_{k_1}\lambda_{i_1}^{(k_1)}>0$, $t_{k_2}\lambda_{i_2}^{(k_2)}>0$. By \eqref{K-T}, we obtain $\bar{x}_{k_1}\in \partial \mathcal{D}_{i_1}(x_{P_{i_1}},x_E), \bar{x}_{k_2}\in \partial \mathcal{D}_{i_2}(x_{P_{i_2}},x_E)$. 
	Let
	\begin{align*}
		\hat{x}_{k_1}^{(i)}=&x_E+\rho_i\left(x_{P_i},x_E,\frac{\bar{x}_{k_1}-x_E}{ \left\| \bar{x}_{k_1}-x_E\right\|}  \right)\frac{\bar{x}_{k_1}-x_E}{ \left\| \bar{x}_{k_1}-x_E\right\|},\\
		&i=1,\dots,m,
	\end{align*}
	It is easy to verify that 
	\begin{align}\label{eq_same_direction}
		\frac{\hat{x}_{k_1}^{(i)}-x_E}{ \left\| \hat{x}_{k_1}^{(i)}-x_E\right\|}=\frac{\bar{x}_{k_1}-x_E}{ \left\| \bar{x}_{k_1}-x_E\right\|}.
	\end{align}
	Let
	\begin{align*}
		\mathbf{z}=
		\begin{pmatrix}
			\frac{\alpha_1\left( \hat{x}_{k_1}^{(1)}-x_{P_1}\right) }{\left\|\hat{x}_{k_1}^{(1)}-x_{P_1} \right\| }\\
			\vdots\\
			\frac{\alpha_m\left( \hat{x}_{k_1}^{(m)}-x_{P_m}\right) }{\left\|\hat{x}_{k_1}^{(m)}-x_{P_m} \right\| }\\[10pt]
			\frac{ \bar{x}_{k_1}-x_E}{\left\| \bar{x}_{k_1}-x_E \right\| }
		\end{pmatrix}
	\end{align*}
	Consider $\mathbf{p}^\top \mathbf{z}$, see \eqref{pz} at the top of the page. 
	When $t_k\lambda_i^{(k)}=0$, we have \eqref{=0} which is shown at the top of the page. 
	When $t_k\lambda_i^{(k)}>0$, by \eqref{eq_same_direction} and \autoref{key_C-oval_property}, we have \eqref{>=0} which is shown at the top of the page. Due to $\bar{x}_{k_1}\neq\bar{x}_{k_2}$, we have $\hat{x}_{k_1}^{(i_2)}\neq\bar{x}_{k_2}$. Then, by the equality condition of \autoref{key_C-oval_property}, we have \eqref{>0} which is shown at the top of the page. Thus $\mathbf{p}^\top \mathbf{z}>0$.\\
	Consider $V^g(\mathbf{y}+\epsilon\mathbf{z})$ where $\epsilon\in(0,\left\| \bar{x}_{k_1}-x_E\right\| )$. According to \autoref{th_epsilon_lemma}(b), $\mathcal{D}^*(\mathbf{y}+\epsilon\mathbf{z})\subseteq\mathcal{D}^*(\mathbf{y})$. Then, $V^g(\mathbf{y}+\epsilon\mathbf{z})\ge V^g(\mathbf{y})$.
	By \autoref{th_epsilon_lemma}(a), for any $i$,
	\begin{align*}
		\hat{x}_{k_1}^{(i)}\in \partial \mathcal{D}_i\left( x_{P_i}+\epsilon \alpha_i \frac{\hat{x}_{k_1}^{(i)}-x_{P_i}}{\left\| \hat{x}_{k_1}^{(i)}-x_{P_i}\right\| },x_E+\epsilon \frac{\bar{x}_{k_1}-x_E}{\left\| \bar{x}_{k_1}-x_E\right\| }\right). 
	\end{align*}
	We obtain \eqref{v(x+ez)=v(x)--1} which is shown at the top of the page. By \eqref{v(x+ez)=v(x)--1}, we derive \eqref{v(x+ez)=v(x)--3} which is also shown at the top of the page. By \eqref{v(x+ez)=v(x)--3}, we obtain $\bar{x}_{k_1}\in\partial\mathcal{D}^*(\mathbf{y}+\epsilon\mathbf{z})$. Then $V^g(\mathbf{y}+\epsilon\mathbf{z})\le g(\bar{x}_{k_1})=V^g(\mathbf{y})$.  
	We have $V^g(\mathbf{y}+\epsilon\mathbf{z})=V^g(\mathbf{y})$ when $\epsilon\in\left( 0,\left\| \bar{x}_{k_1}-x_E\right\|\right)  $. Therefore,
	\begin{align*}
		\lim\limits_{\epsilon\rightarrow0^+}\frac{V^g(\mathbf{y}+\epsilon\mathbf{z})-V^g(\mathbf{y})-\mathbf{p}^\top\epsilon\mathbf{z}}{\left\| \epsilon\mathbf{z}\right\| }
		=-\frac{\mathbf{p}^\top\mathbf{z}}{\left\| \mathbf{z}\right\|}
		<0.
	\end{align*}
	This contradicts that $\mathbf{p}\in D^-V^g(\mathbf{y})$ by the definition of subdifferential. We obtain that $\left\lbrace \bar{x}_k:\exists i,t_k\lambda_i^{(k)}>0\right\rbrace$ is either a singleton or an empty set. If $\left\lbrace \bar{x}_k:\exists i,t_k\lambda_i^{(k)}>0\right\rbrace$ is an empty set,
	then $\forall i$, $\forall k$, $t_k\lambda_i^{(k)}=0$. Thus,
	\begin{align*}
		p_{P_1}=p_{P_2}=\dots=p_{P_m}=p_E=0.
	\end{align*}
	\eqref{eq_super} is satisfied. If $\left\lbrace \bar{x}_k:\exists i,t_k\lambda_i^{(k)}>0\right\rbrace$ is a singleton, then let $\left\lbrace \bar{x}_k:\exists i,t_k\lambda_i^{(k)}>0\right\rbrace$=$\left\lbrace \bar{x}^*\right\rbrace$ and $\hat{K}=\left\lbrace k:\exists i,t_k\lambda_i^{(k)}>0\right\rbrace$. We have
	\begin{align*}
		p_{P_i}=& \sum_{k\in\hat{K}} t_k8\alpha_i l_i \lambda_i^{(k)} \left\| \bar{x}^*-x_E\right\| (\bar{x}^*-x_{P_i}), \\
		&i=1,\dots,m,\\
		p_{E}=&- \sum_{k\in\hat{K}} t_k\sum_{i=1}^m8\alpha_i^2 l_i \lambda_i^{(k)} \left\| \bar{x}^*-x_{P_i}\right\| (\bar{x}^*-x_E).
	\end{align*}
	Then,
	\begin{align*}
		&\sum_{i=1}^m\alpha_i\left\| {p}_{P_i}\right\| \\
		=& \sum_{i=1}^m \sum_{k\in\hat{K}} t_k8\alpha_i^2 l_i \lambda_i^{(k)} \left\| \bar{x}^*-x_E\right\|\left\| \bar{x}^*-x_{P_i}\right\|  \\
		=&\left\| {p}_E\right\|. 
	\end{align*}
	\eqref{eq_super} is satisfied. The proof is completed.
\end{proof}

\subsection{Summary}\label{subsec_Summary}
At the end of this section,  we summarize the main result of this paper.
\begin{theorem}\label{th_sum_viscosity}
	Assume that $\alpha_i>1$, $l_i>0$ for any $i\in[m]$ and $g:\mathbb{R}^n\rightarrow \mathbb{R}$ is locally Lipschitz continuous on $\mathbb{R}^n$. Then, $V^g$ is a viscosity solution of the HJI PDE Dirichlet problem \eqref{HJI}.
\end{theorem}
\begin{proof}
	That $V^g$ is a viscosity solution of the HJI PDE Dirichlet problem \eqref{HJI} is equivalent to that $V^g$ satisfies all of the following statements:\\
	(a) $V^g$ is the viscosity solution of the HJI PDE of \eqref{HJI};\\
	(b) $V^g$ satisfies the boundary condition of \eqref{HJI};\\
	(c) $V^g$ is continuous in $\Omega\cup\partial\mathcal{T}$.
	
	By combining \autoref{th_sub2} and \autoref{th_super}, we obtain (a). Let $\mathbf{y}$ denote $(x_{P_1}^\top,\dots,x_{P_m}^\top,x_E^\top)^\top$.
	When $\mathbf{y}\in\partial\mathcal{T}$, it is obvious that $\mathcal{D}^*(\mathbf{y})=\left\lbrace x_E\right\rbrace $. This implies that $\inf_{x\in\mathcal{D}^*(\mathbf{y})}g(x)=g(x_E)$. Thus, the boundary condition in \eqref{HJI} holds. (b) is obtained.
	
	Pick $\mathbf{y}^*=({x_{P_1}^*}^\top,\dots,{x_{P_m}^*}^\top,{x_E^*}^\top)^\top\in\partial\mathcal{T}$. According to the continuity of $g$, for $\epsilon>0$ and $x_E^*$, there exist $\delta_1>0$ such that $\left| g(x)-g(x_E^*)\right|\le \epsilon $ for any $x\in B(x_E^*,\delta_1)$. In the proof of \autoref{th_uniform_compactness}, we show that $\forall \mathbf{y}\in\Omega\cup\partial\mathcal{T}$, $\forall x\in \mathcal{D}^*(\mathbf{y})$,
	$\left\| x-x_E \right\| \le M(\mathbf{y})$, 
	where $M(\mathbf{y})=\min_{i} \frac{\left\| x_{P_i}-x_E \right\| -l_i}{\alpha_i-1}$. According to the continuity of $M(\mathbf{y})$, for $\delta_1$ and $\mathbf{y}^*$, there exist $\delta_2>0$ such that $M(\mathbf{y})=\left|M(\mathbf{y})-M(\mathbf{y}^*) \right|<\delta_1 /2 $ for any $\mathbf{y}\in B(\mathbf{y}^*,\delta_2)$. Then, for any $\mathbf{y}\in B(\mathbf{y}^*,\min\left\lbrace \delta_1 /2, \delta_2\right\rbrace )\cap(\Omega\cup\partial\mathcal{T})$, for any $x\in\mathcal{D}^*(\mathbf{y})$,
	\begin{align*}
		&\left\| x-x_E^*\right\| 
		\le \left\| x-x_E\right\| +\left\| x_E-x_E^*\right\|\\
		\le& M(\mathbf{y})+\delta_1/2
		\le\delta_1/2+\delta_1/2
		=\delta_1.
	\end{align*}
	Then, for any $\mathbf{y}\in B(\mathbf{y}^*,\min\left\lbrace \delta_1 /2, \delta_2\right\rbrace )\cap(\Omega\cup\partial\mathcal{T})$, for any $x\in\mathcal{D}^*(\mathbf{y})$, $\left| g(x)-g(x_E^*)\right| < \epsilon $. This implies $\left| V^g(\mathbf{y})-V^g(\mathbf{y}^*)\right| < \epsilon $, $\forall\mathbf{y}\in B(\mathbf{y}^*,\min\left\lbrace \delta_1 /2, \delta_2\right\rbrace )\cap(\Omega\cup\partial\mathcal{T})$. We derive the continuity of $V^g$ at $\mathbf{y}^*$. Combining this and \autoref{estimate of generalized gradient}, we obtain $V^g$ is continuous in $\Omega\cup\partial\mathcal{T}$. (c) is obtained. The proof is completed.
\end{proof}

From \cite[Corollary 2.8]{soravia1993pursuit} and \autoref{th_sum_viscosity}, we derive the following result.
\begin{theorem}\label{th_soravia}
	Assume that $\alpha_i>1$, $l_i>0$ for any $i\in[m]$ and $g:\mathbb{R}^n\rightarrow \mathbb{R}$ is locally Lipschitz continuous on $\mathbb{R}^n$. Then,
	\begin{align*}
		&V^g(\mathbf{y})=\\
		&\sup_{(\delta_{P_1},\dots,\delta_{P_m})\in \Delta_{\mathbf{y}}}
		\inf_{u_E\in\mathcal{U}_0} J(\mathbf{y},\delta_{P_1}[u_E],\dots,\delta_{P_m}[u_E],u_E),
	\end{align*}
	where $\Delta_{\mathbf{y}}$ is the set of non-anticipative strategies \cite{elliott1972existence,cardaliaguet1996differential,soravia1993pursuit,mitchell2005time} from $\mathcal{U}_0$ to $\mathcal{U}_0^m$ with $\mathbf{y}=(x_{P_1}^\top,\dots,x_{P_m}^\top,x_E^\top)^\top$ as the initial state.
\end{theorem}

\begin{remark}
	In this study, we require the capture radii of all pursuers to be positive. When the capture radii of some or all pursuers are zero, $V^g$ remains the viscosity solution of the HJI PDE in \eqref{HJI}. The proofs are similar. Due to space limitations, we do not elaborate on this further. It should be  emphasized that the constraint 
	\begin{align*}
		-((x-x_{P_i})^2-\alpha_i^2(x-x_E)^2-l_i^2)^2+4\alpha_i^2 l_i^2 (x-x_E)^2\le 0
	\end{align*}
	is not suitable for $l_i=0$, because the Mangasarian-Fromovitz regularity condition is not satisfied. The constraint 
	\begin{align*}
		-\left\| x-x_{P_i}\right\| +\alpha_i\left\| x-x_E\right\| \le 0
	\end{align*}
	can be directly changed to
	\begin{align*}
		-(x-x_{P_i})^2+\alpha_i^2(x-x_E)^2\le 0.
	\end{align*}
	The other steps of proofs are identical.
\end{remark}

\section{Optimal strategies}\label{sec_optimal_strategies}
In this section, we discuss the optimal strategies for the quantitative game. 

In the previous studies \cite{fu2023justification,yan2022matching,lee2024solutions}, the authors presented a group of equilibrium strategies for the game. In their studies, the terminal cost function $g$ with respect to the evader's position is convex, which guarantees that the optimal solution of the mathematical program \eqref{program} is unique. The equilibrium strategies are as follows,
\begin{equation}\label{eq_optimal_feedback}
	\begin{split}
		u_{P_i}&=\frac{x_*^{(i)}-x_{P_i}}{\left\|x_*^{(i)}-x_{P_i}\right\|},i=1,\dots,m,\\
		u_E&=\frac{x_*-x_E}{\left\|x_*-x_E \right\| },
	\end{split}
\end{equation}
where $x_*=x_*(x_{P_1},\dots,x_{P_m},x_E)$ is the unique optimal solution of the mathematical program \eqref{program} and
\begin{align*}
	x_*^{(i)}&=x_*^{(i)}(x_{P_1},\dots,x_{P_m},x_E)\\
	&\triangleq x_E+\rho_i\left(  x_{P_i},x_E,\frac{x_*-x_E}{\left\|x_*-x_E \right\| }\right) \frac{x_*-x_E}{\left\|x_*-x_E \right\| }.
\end{align*}
These are state feedback strategies.
When the terminal cost function is not convex, the uniqueness of the optimal solution is not guaranteed. Thus, the state feedback strategies \eqref{eq_optimal_feedback} are not well-defined. Equilibrium state feedback strategies may not exist (see \cite[Appendix]{dorothy2024one}).  

There is a simple and straightforward method to overcome this problem mathematically. We allow the pursuers have the access to the evader's current instantaneous control. Consider the pursuit strategies \eqref{eq_prusuer_current_input}. 
On the one hand, when the pursuers use the strategies \eqref{eq_prusuer_current_input}, according to \autoref{th_pursuit_strategy}, $\mathcal{D}_i(x_{P_i}(t),x_E(t))$ shrinks inward, regardless of the evader's control input function. Because $x_E(t)$ is always in $\mathcal{D}_i(x_{P_i}(t),x_E(t))$, the evader can not leave $\mathcal{D}^*(x_{P_1}^0,\dots,x_{P_m}^0,x_E^0)$ when the pursuers use \eqref{eq_prusuer_current_input}. 
On the other hand, it is well-known that
\begin{align*}
	\mathcal{D}_i(x_{P_i}^0,x_E^0)^\circ=\left\lbrace x\in \mathbb{R}^n :\exists u_E\in \mathcal{U}_0,\forall u_{P_i}\in\mathcal{U}_0,\right.\\
	\left.\exists t\in [t_0,t_f^i(u_{P_i},u_E)) ,x_E(t;x_E^0,u_E)=x\right\rbrace.
\end{align*}
The evader can reach any point in $\mathcal{D}^*(x_{P_1}^0,\dots,x_{P_m}^0,x_E^0)^\circ$ before being captured by any pursuer, regardless of  the pursuers' control input function.
Thus, the pursuers' optimal strategies are \eqref{eq_prusuer_current_input}, while the evader's optimal strategy is moving toward one of the optimal solutions of \eqref{program} by its maximum speed until the evader reaches it and staying there. They constitute a group of equilibrium strategies in the non-anticipative information pattern \cite{bardi1997optimal}. The above statement can be seen as an intuitive explanation of that $V^g$ is the value function of the quantitative game.

Similarly to mass point and Isaacs' simple motion, the pursuers' access to the evader's current instantaneous control is an idealized mathematical assumption. For practical applications, this access may not be allowed. We can approximate this by the delayed availability of $u_E$, i.e. replacing $u_E(t)$ by $u_E(t-\Delta t)$ $(\Delta t >0)$. The evaluation of this approximation is left in our future works.

In \cite{dorothy2024one}, the authors presented a group of $\epsilon$-equilibrium state feedback strategies for this quantitative game when capture radii were zero. See \cite{dorothy2024one} for details.

\section{Applications}\label{sec_application}
In this section, we discuss the application of our work in target defense games to present the research motivation.

We continue to use the concepts and symbols defined in previous sections. Let $T$ be a nonempty closed subset of $\mathbb{R}^n$. $T$ is called target. We assume that $T$ and the locally Lipschitz function $g:\mathbb{R}^n\rightarrow\mathbb{R}$ satisfy
\begin{align}\label{target_lip}
	T=\left\lbrace x\in\mathbb{R}^n:g(x)\le 0\right\rbrace. 
\end{align}
In the target defense game, the evader intends to enter the target before being captured, while the pursuers seek to capture the evader before the evader enters the target.

The problem is under what initial state the pursuers have non-anticipative strategies to prevent the evader entering the target before being captured regardless of the evader’s strategy. This is equivalent to determining the following set,
\begin{align*}
	W_P\triangleq \lbrace \mathbf{y}^0\in \Omega:\exists (\delta_{P_1},\dots,\delta_{P_m})\in \Delta_{\mathbf{y}^0}, \\
	\forall u_E(\cdot) \in \mathcal{U}, \forall t\in[t_0,t_f], x_E(t;x_E^0,u_E)\notin T^\circ \rbrace,
\end{align*} 
where $\mathbf{y}^0=(x_{P_1}^{0\top},\dots,x_{P_m}^{0\top},x_E^{0\top})^\top$ and $\Delta_{\mathbf{y}^0}$ is the set of non-anticipative strategies \cite{elliott1972existence,cardaliaguet1996differential,soravia1993pursuit,mitchell2005time} from $\mathcal{U}_0$ to $\mathcal{U}_0^m$ with $\mathbf{y^0}$ as the initial state. $W_P$ is called the winning set of the pursuers. 

We can determine $W_P$ by $V^g$, as shown in following theorem.
\begin{theorem}
	Consider the target defense game defined above. Assume that $\alpha_i>1,l_i\ge 0,\forall i\in[m]$ and $g$ is locally Lipchitz continuous on $\mathbb{R}^n$ and \eqref{target_lip} holds. Then, 
	\begin{align*}
		W_P=\left\lbrace \mathbf{y}^0\in\Omega:V^g(\mathbf{y}^0)\ge 0\right\rbrace.
	\end{align*}
\end{theorem}
\begin{proof}
	We have shown that the optimal strategies exist in Section \ref{sec_optimal_strategies}. According to \eqref{target_lip} and \autoref{th_soravia}, we have
	\begin{align*}
		&\mathbf{y}^0\in W_P\\
		\iff& \exists (\delta_{P_1},\dots,\delta_{P_m})\in \Delta_{\mathbf{y}^0}, \forall u_E(\cdot) \in \mathcal{U},\\
		&\forall t\in[t_0,t_f], g(x_E(t;x_E^0,u_E))\ge0\\
		\Longrightarrow& \exists (\delta_{P_1},\dots,\delta_{P_m})\in \Delta_{\mathbf{y}^0}, \forall u_E(\cdot) \in \mathcal{U},\\
		&g(x_E(t_f;x_E^0,u_E))\ge 0\\
		\iff& \exists (\delta_{P_1},\dots,\delta_{P_m})\in \Delta_{\mathbf{y}^0}, \forall u_E(\cdot) \in \mathcal{U},\\
		& J(\mathbf{y}^0,\delta_{P_1}[u_E],\dots,\delta_{P_m}[u_E],u_E)\ge0\\
		\iff& \exists (\delta_{P_1},\dots,\delta_{P_m})\in \Delta_{\mathbf{y}^0},\\
		& \inf_{u_E\in\mathcal{U}}J(\mathbf{y}^0,\delta_{P_1}[u_E],\dots,\delta_{P_m}[u_E],u_E)\ge0\\
		\iff& \max_{(\delta_{P_1},\dots,\delta_{P_m})\in \Delta_{\mathbf{y}}}
		\inf_{u_E\in\mathcal{U}_0} J(\mathbf{y}^0,\delta_{P_1}[u_E],\dots,\delta_{P_m}[u_E],u_E)\\
		&\ge0\\
		\iff& V^g(\mathbf{y}^0)\ge0
	\end{align*}
	Then, we have $W_P\subseteq \left\lbrace \mathbf{y}^0\in\Omega:V^g(\mathbf{y}^0)\ge0\right\rbrace$. To deserve the equality, we only need to prove that 
	\begin{align*}
		& \exists (\delta_{P_1},\dots,\delta_{P_m})\in \Delta_{\mathbf{y}^0}, \forall u_E(\cdot) \in \mathcal{U},\\
		&g(x_E(t_f;x_E^0,u_E))\ge0\\
		\Longrightarrow&\exists (\delta_{P_1},\dots,\delta_{P_m})\in \Delta_{\mathbf{y}^0}, \forall u_E(\cdot) \in \mathcal{U},\\
		&\forall t\in[t_0,t_f], g(x_E(t;x_E^0,u_E))\ge0
	\end{align*}
	It is equivalent to proving its contrapositive:
	\begin{align*}
		&\forall (\delta_{P_1},\dots,\delta_{P_m})\in \Delta_{\mathbf{y}^0}, \exists u_E(\cdot) \in \mathcal{U}, \\
		&\exists t\in[t_0,t_f], g(x_E(t;x_E^0,u_E))< 0\\
		\Longrightarrow &\forall (\delta_{P_1},\dots,\delta_{P_m})\in \Delta_{\mathbf{y}^0}, \exists u_E(\cdot) \in \mathcal{U}, \\
		&g(x_E(t_f;x_E^0,u_E))< 0.
	\end{align*}
	Let $(\delta_{P_1},\dots,\delta_{P_m})\in \Delta_{\mathbf{y}^0}$ and $u_E(\cdot) \in \mathcal{U}$ satisfy $\exists \bar{t}\in(t_0,t_f(\delta_{P_1}[u_E],\dots,\delta_{P_m}[u_E],u_E);x_E^0,u_E)]$, $ g(x_E(\bar{t};x_E^0,u_E))< 0$. Let
	\begin{align*}
		\bar{u}_E(t)=
		\begin{cases}
			u_E(t), &t\in[t_0,\bar{t}]\\
			0,  &t\in(\bar{t},+\infty)
		\end{cases}
	\end{align*}
	Obviously, $\bar{u}_E(\cdot)\in \mathcal{U}$ and $\forall t\in [\bar{t},+\infty),x_E(t;x_E^0,\bar{u}_E)=x_E(\bar{t};x_E^0,u_E)$. Thus,
	\begin{align*}
		&g(x_E(t_f(\delta_{P_1}[\bar{u}_E],\dots,\delta_{P_m}[\bar{u}_E],\bar{u}_E);x_E^0,\bar{u}_E))\\
		=&g(x_E(\bar{t};x_E^0,u_E))< 0,
	\end{align*}
	The proof is completed.
\end{proof}

Given the initial positions of the pursuers and the evader, we can determine whether the pursuers can win the game by calculating $V^g$ (i.e. the optimal value of the parametric mathematical programming problem \eqref{modified_program}) rather than solving the HJI PDE Dirichlet problem \eqref{HJI} according to \autoref{th_sum_viscosity}.

At the beginning of this section, we assume that the target can be represented by a locally Lipschitz function as \eqref{target_lip}. This assumption is reasonable. For any nonempty set $A\subseteq\mathbb{R}^n$, we define the following function:
\begin{align*}
	\tilde{d}_A(x)\triangleq
	\begin{cases}
		-\inf_{y\in \partial A}\left\| x-y\right\|, &x\in A\\
		\inf_{y\in \partial A}\left\| x-y\right\|, &x\notin A
	\end{cases}
\end{align*}
$\tilde{d}_A$ is called the signed distance function of $A$. It is obvious that $A=\left\lbrace x\in \mathbb{R}^n:d_A(x)\le 0\right\rbrace$ if $A$ is a closed subset of $\mathbb{R}^n$. Below is a well-known theorem about signed distance functions.
\begin{theorem}\label{signed_distance_lip}
	For any set $A$ contained in $\mathbb{R}^n$ (where $A$ is not requested to be close), the signed distance function $\tilde{d}_A:\mathbb{R}^n\rightarrow\mathbb{R}$ is Lipschitz continuous.
\end{theorem}
The proof is presented in the supplementary material.
According to \autoref{signed_distance_lip}, any closed set contained by $\mathbb{R}^n$ can be represented by a locally Lipschitz continuous function as \eqref{target_lip}. Thus, our work covers any form of target. By comparison with it, the works of \cite{fu2023justification,lee2024solutions} only consider the case where the target is convex.
\begin{remark}
	Using the signed distance function is not the only way to represent the target by a locally Lipschitz continuous function as \eqref{target_lip}. For example, consider $T=\left\lbrace (y_1,y_2)^\top \in \mathbb{R}^2:y_1^2-y_2\le0\right\rbrace $. It is obvious that $\tilde{d}_T(y_1,y_2)\neq y_1^2-y_2$.
\end{remark}
\section{Conclusion}\label{sec_conclusion}
In this study, we investigate a multiple-pursuer single-evader quantitative pursuit-evasion game. The payoff function of the game includes only the terminal cost which is related only to the evader's terminal position. We verify that a candidate for the value generated by the geometric method is the viscosity solution of the corresponding HJI PDE Dirichlet problem, without requiring the convexity of the terminal cost. We also discuss the optimal strategies for the game. To extend the current study, we will consider the multiple-pursuer multiple-evader pursuit-evasion game in future research.

\section*{References}
\bibliographystyle{IEEEtran}
\bibliography{autosam}



\section*{Supplementary material}

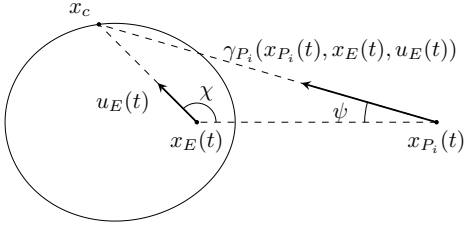

Fig. 1: $\chi$ and $\psi$ in the proof of Proposition 3.2(a).

*Proof:* [proof of Proposition 3.2] (a) Fix $t$. If $u_E(t) = 0$, then

$$\frac{d}{dt}\|x_{P_i}(t) - x_E(t)\|$$
$$= \frac{(x_{P_i}(t) - x_E(t))^\top}{\|x_{P_i}(t) - x_E(t)\|} \frac{\alpha_i(x_E(t) - x_{P_i}(t))}{\|x_E(t) - x_{P_i}(t)\|}$$
$$= -\alpha_i < 1 - \alpha_i.$$

Consider $u_E(t) \neq 0$. Let $\chi$ denote the angle between vectors $u_E(t)$ and $x_{P_i}(t) - x_E(t)$, and let $\psi$ denote the angle between vectors $\gamma_{P_i}(x_{P_i}(t), x_E(t), u_E(t))$ and $x_E(t) - x_{P_i}(t)$, $\chi, \psi \in [0, \pi]$, as shown in Fig. 1. By the law of sines,

$$\frac{\sin \chi}{\sin \psi} = \frac{\alpha_i \rho_i\left(x_{P_i}, x_E, \frac{u_E}{\|u_E\|}\right) + l_i}{\rho_i\left(x_{P_i}, x_E, \frac{u_E}{\|u_E\|}\right)} \geq \alpha_i.$$

Then, $\cos \psi \geq \sqrt{1 - \frac{1}{\alpha_i^2}\sin^2 \chi}$. We obtain that

$$\frac{d}{dt}\|x_{P_i}(t) - x_E(t)\|$$
$$= -\alpha_i \cos \psi - \cos \chi$$
$$\leq -\alpha_i \sqrt{1 - \frac{1}{\alpha_i^2}\sin^2 \chi} - \cos \chi.$$

Let $f(\chi) = -\alpha_i \sqrt{1 - \frac{1}{\alpha_i^2}\sin^2 \chi} - \cos \chi$.

$$f'(\chi) = \left(\frac{\cos \chi}{\sqrt{\alpha_i^2 - \sin^2 \chi}} + 1\right)\sin \chi$$
$$> \frac{\cos \chi + |\cos \chi|}{\sqrt{\alpha_i^2 - \sin^2 \chi}}\sin \chi \geq 0.$$

Then,

$$\frac{d}{dt}\|x_{P_i}(t) - x_E(t)\| \leq f(\pi) = 1 - \alpha_i.$$

(b) Fix $t_1, t_2$. Pick a point $x_d$ from $\partial \mathcal{D}_i(x_{P_i}(t_1), x_E(t_1))$. Take the tangent hyperplane of $\partial \mathcal{D}_i(x_{P_i}(t_1), x_E(t_1))$ at $x_d$. Due to the convexity of $\mathcal{D}_i(x_{P_i}(t_1), x_E(t_1))$, the whole of $\mathcal{D}_i(x_{P_i}(t_1), x_E(t_1))$ is located on the same side of the hyperplane. $D_x d_i(x_d, x_{P_i}(t_1), x_E(t_1))$ is the outer normal vector of $\partial \mathcal{D}_i(x_{P_i}(t_1), x_E(t_1))$ at $x_d$. The side where $\mathcal{D}_i(x_{P_i}(t_1), x_E(t_1))$ is located can be represented as follows,

$$\Sigma(x_d) \triangleq \{x \in \mathbb{R}^n : D_x d_i(x_d, x_{P_i}(t_1), x_E(t_1))^\top (x - x_d) \leq 0\}$$

For any $t \in [t_1, t_2)$, consider the mathematical program as follows,

$$\max_{x \in \mathbb{R}^n} D_x d_i(x_d, x_{P_i}(t_1), x_E(t_1))^\top (x - x_d)$$
$$\text{s.t. } d_i(x; x_{P_i}(t), x_E(t)) \leq 0$$

The set of the feasible solutions is $\mathcal{D}_i(x_{P_i}(t), x_E(t))$ which is bounded, closed and strictly convex. $D_x d_i(x_d, x_{P_i}(t_1), x_E(t_1))^\top (x - x_d)$ is an affine function with respect to $x$. Thus, the optimal solution of this mathematical program exists and is unique, which is denoted by $x_H(t) = x_H(x_{P_i}(t), x_E(t))$. The optimal solution is also on $\mathcal{D}_i(x_{P_i}(t), x_E(t))$, i.e.

$$d_i(x_H(t); x_{P_i}(t), x_E(t)) = 0 \tag{C.5}$$

By Karush-Kuhn-Tucker conditions, $x_H(t)$ satisfies

$$D_x d_i(x_d, x_{P_i}(t_1), x_E(t_1))$$
$$= \lambda D_x d_i(x_H(t), x_{P_i}(t), x_E(t)), \tag{C.6a}$$
$$\lambda d_i(x_H(t), x_{P_i}(t_1), x_E(t_1)) = 0, \tag{C.6b}$$
$$\lambda \geq 0 \tag{C.6c}$$

Take the derivatives of both sides of the (C.5) with respect to $t$,

$$D_x d_i(x_H(t), x_{P_i}(t), x_E(t))^\top \dot{x}_H(t)$$
$$+ \frac{(x_H(t) - x_{P_i}(t))^\top}{\|x_H(t) - x_{P_i}(t)\|}\dot{x}_{P_i}(t) \tag{C.7}$$
$$- \alpha_i \frac{(x_H(t) - x_E(t))^\top}{\|x_H(t) - x_E(t)\|}\dot{x}_E(t) = 0$$

Multiplying both sides of (C.7) by $\lambda$ and substituting (31) and (C.6a) into it, we obtain

$$D_x d_i(x_d, x_{P_i}(t_1), x_E(t_1))^\top \dot{x}_H(t)$$
$$= \lambda \alpha_i \left(\frac{(x_H(t) - x_E(t))^\top}{\|x_H(t) - x_E(t)\|}u_E(t)\right. \tag{C.8}$$
$$\left. - \frac{(x_H(t) - x_{P_i}(t))^\top}{\|x_H(t) - x_{P_i}(t)\|}\gamma_{P_i}(x_{P_i}(t), x_E(t), u_E(t))\right)$$

When $u_E = 0$, by Lemma C.1,

$$D_x d_i(x_d, x_{P_i}(t_1), x_E(t_1))^\top \dot{x}_H(t)$$
$$= -\lambda \alpha_i \frac{(x_H(t) - x_{P_i}(t))^\top}{\|x_H(t) - x_{P_i}(t)\|}\frac{(x_E(t) - x_{P_i}(t))}{\|x_E(t) - x_{P_i}(t)\|} < 0$$

When $u_E \neq 0$, let

$$x_c = x_c(x_{P_i}(t), x_E(t), u_E(t))$$
$$\triangleq x_E(t) + \rho_i\left(x_{P_i}(t), x_E(t), \frac{u_E(t)}{\|u_E(t)\|}\right)\frac{u_E(t)}{\|u_E(t)\|}.$$

We have

$$u_E(t) = \frac{\|u_E(t)\|(x_c - x_E(t))}{\|x_c - x_E(t)\|},$$
$$\gamma_{P_i}(x_{P_i}(t), x_E(t), u_E(t)) = \frac{x_c - x_{P_i}(t)}{\|x_c - x_{P_i}(t)\|}.$$

Substitute them into (C.8), we obtain

$$D_x d_i(x_d, x_{P_i}(t_1), x_E(t_1))^\top \dot{x}_H(t)$$
$$= \lambda \alpha_i \left( \frac{(x_H(t) - x_E(t))^\top}{\|x_H(t) - x_E(t)\|} \frac{\|u_E(t)\|(x_c - x_E(t))}{\|x_c - x_E(t)\|} \right.$$
$$\left. - \frac{(x_H(t) - x_{P_i}(t))^\top}{\|x_H(t) - x_{P_i}(t)\|} \frac{x_c - x_{P_i}(t)}{\|x_c - x_{P_i}(t)\|} \right)$$

If $(x_H(t) - x_E(t))^\top (x_c - x_E(t)) \leq 0$, according to Lemma C.1, $D_x d_i(x_d, x_{P_i}(t_1), x_E(t_1))^\top \dot{x}_H(t) < 0$. If $(x_H(t) - x_E(t))^\top (x_c - x_E(t)) > 0$, by Proposition 3.1, we derive

$$D_x d_i(x_d, x_{P_i}(t_1), x_E(t_1))^\top \dot{x}_H(t)$$
$$\leq \lambda \alpha_i \left( \frac{(x_H(t) - x_E(t))^\top}{\|x_H(t) - x_E(t)\|} \frac{(x_c - x_E(t))}{\|x_c - x_E(t)\|} \right.$$
$$\left. - \frac{(x_H(t) - x_{P_i}(t))^\top}{\|x_H(t) - x_{P_i}(t)\|} \frac{x_c - x_{P_i}(t)}{\|x_c - x_{P_i}(t)\|} \right)$$
$$\leq 0.$$

In summary, $D_x d_i(x_d, x_{P_i}(t_1), x_E(t_1))^\top \dot{x}_H(t) \leq 0$. Noting that $x_d = x_H(t_1)$, we have

$$D_x d_i(x_d, x_{P_i}(t_1), x_E(t_1))^\top (x_H(t_2) - x_d)$$
$$= \int_{t_1}^{t_2} D_x d_i(x_d, x_{P_i}(t_1), x_E(t_1))^\top \dot{x}_H(t) dt$$
$$\leq 0.$$

According to the definition of $x_H(t)$, $\mathcal{D}_i(x_{P_i}(t_2), x_E(t_2)) \subseteq \Sigma(x_d)$. Noting that $x_d$ is picked arbitrarily, we obtain

$$\mathcal{D}_i(x_{P_i}(t_2), x_E(t_2)) \subseteq \bigcup_{x \in \partial \mathcal{D}_i(x_{P_i}(t_1), x_E(t_1))} \Sigma(x)$$
$$= \mathcal{D}_i(x_{P_i}(t_1), x_E(t_1)),$$

where the last equality holds by [2, Theorem 18.8]. The proof is completed. ■

### D. PROOF OF PROPOSITION 4.2

*Proof:* Select $\mathbf{y} = (x_{P_1}^\top, \ldots, x_{P_m}^\top, x_E^\top)^\top \in \Omega$ and $x \in \mathcal{P}(\mathbf{y})$. By "(b) $\iff$ (c)" in Proposition 4.1, $\{i : \bar{d}_i(x; x_{P_i}, x_E) = 0\}$ is empty set. For any $i \in \hat{I}(x, \mathbf{y}) = \{i : \hat{d}_i(x; x_{P_i}, x_E) = 0\}$,

$$(x_E - x)^\top D_x \hat{d}_i(x; x_{P_i}, x_E)$$
$$= (x_E - x)^\top (-4((x - x_{P_i})^2 - \alpha_i^2(x - x_E)^2 - l_i^2)(x - x_{P_i})$$
$$+ 4\alpha_i^2((x - x_{P_i})^2 - \alpha_i^2(x - x_E)^2 - l_i^2 + 2l_i^2)(x - x_E)). \tag{D.9}$$

It is easy to obtain that

$$\hat{d}_i(x; x_{P_i}, x_E) = 0$$

implies

$$-((x - x_{P_i})^2 - \alpha_i^2(x - x_E)^2 - l_i^2) = -2\alpha_i l_i \|x - x_E\|. \tag{D.10}$$

Substituding (D.10) into (D.9), we obtain

$$(x_E - x)^\top D_x \hat{d}_i(x; x_{P_i}, x_E)$$
$$= (x_E - x)^\top 8\alpha_i l_i (- \|x - x_E\|(x - x_{P_i})$$
$$+ \alpha \|x - x_{P_i}\|(x - x_E))$$
$$\leq 8\alpha_i l_i \|x - x_E\|^2 \|x - x_{P_i}\|(1 - \alpha_i)$$
$$< 0.$$

The proof is completed. ■

### E. DIFFERENCE OF VERIFICATION BETWEEN OUR WORK AND THE PREVIOUS WORKS

In this part, we compare the verification of the case in which $g$ is locally Lipschitz continuous with that of the case in which $g$ is convex. Let us consider the case in which $g$ is convex. Due to the convexity, the optimal solution of (44) is unique. Pick $\mathbf{p} = (p_{P_1}^\top, \ldots, p_{P_m}^\top, p_E^\top)^\top \in \partial V^g(\mathbf{y})$. By Proposition 4.4 and Carathéodory theorem [2], $\mathbf{p}$ can be represented as follows:

$$p_{P_i} = \sum_{k=1}^{n(m+1)+1} t_k 8\alpha_i l_i \lambda_i^{(k)} \|x_* - x_E\|(x_* - x_{P_i}),$$
$$i = 1, \ldots, m,$$
$$p_E = -\sum_{k=1}^{n(m+1)+1} t_k \sum_{i=1}^{m} 8\alpha_i^2 l_i \lambda_i^{(k)} \|x_* - x_{P_i}\|(x_* - x_E),$$

where

$$\{x_*\} = \mathcal{P}(\mathbf{y}),$$
$$(\lambda_1^{(k)}, \ldots, \lambda_m^{(k)}, 0, \ldots, 0) \in \mathcal{K}(x_*, \mathbf{y}),$$
$$t_k \geq 0, k = 1, \ldots, n(m+1)+1,$$
$$\sum_{k=1}^{n(m+1)+1} t_k = 1.$$

Then,

$$\sum_{i=1}^{m} \alpha_i \|p_{P_i}\|$$
$$= \sum_{i=1}^{m} \alpha_i \sum_{k=1}^{n(m+1)+1} t_k 8\alpha_i l_i \lambda_i^{(k)} \|x_* - x_E\| \|x_* - x_{P_i}\|$$
$$= \sum_{k=1}^{n(m+1)+1} t_k \sum_{i=1}^{m} 8\alpha_i^2 l_i \lambda_i^{(k)} \|x_* - x_{P_i}\| \|x_* - x_E\|$$
$$= \|p_E\|$$

From Theorem 2.1, we obtain that $V^g$ is the viscosity solution of the HJI PDE of (4). Proposition 3.1 is not used in the above process and it is not necessary to present the proof of Proposition 4.6. Therefore, the lack of the convexity improves the difficulty of the verification significantly.

### F. PROOF OF PROPOSITION 4.3

*Proof:* By (A.1), when $\mathbf{y} = (x_{P_1}^\top, \ldots, x_{P_m}^\top, x_E^\top)^\top \in \Omega$,

$$\|x - x_E\| \leq \min_i \frac{\|x_{P_i} - x_E\| - l_i}{\alpha_i - 1},$$
$$\forall x \in \mathcal{D}^*(x_{P_1}, \ldots, x_{P_m}, x_E).$$

4

Let

$$M(\mathbf{y}) = \min_i \frac{\|x_{P_i} - x_E\| - l_i}{\alpha_i - 1}.$$

It is obvious that $M(\mathbf{y})$ is continuous. For any $\bar{\mathbf{y}} = (\bar{x}_{P_1}^\top, \ldots, \bar{x}_{P_m}^\top, \bar{x}_E^\top)^\top \in \Omega$, there exists a globe neighborhood of $\bar{\mathbf{y}}$ denoted as $N(\bar{\mathbf{y}})$ such that $N(\bar{\mathbf{y}}) \subseteq \Omega$ and

$$|M(\mathbf{y}) - M(\bar{\mathbf{y}})| < 1, \forall \mathbf{y} \in N(\bar{\mathbf{y}}).$$

For any $x \in \bigcup_{\mathbf{y} \in N(\bar{\mathbf{y}})} \mathcal{D}^*(\mathbf{y})$, there exists $\hat{\mathbf{y}} = (\hat{x}_{P_1}^\top, \ldots, \hat{x}_{P_m}^\top, \hat{x}_E^\top)^\top \in N(\bar{\mathbf{y}})$ such that $x \in \mathcal{D}^*(\hat{\mathbf{y}})$. Then,

$$\|x - \bar{x}_E\| \leq \|x - \hat{x}_E\| + \|\hat{x}_E - \bar{x}_E\|$$
$$\leq M(\hat{\mathbf{y}}) + \delta \leq M(\bar{\mathbf{y}}) + 1 + \delta,$$

where $\delta$ is the radius of $N(\bar{\mathbf{y}})$. $\bigcup_{\mathbf{y} \in N(\bar{\mathbf{y}})} \mathcal{D}^*(\mathbf{y})$ is bounded. Thus, the closure of $\bigcup_{\mathbf{y} \in N(\bar{\mathbf{y}})} \mathcal{D}^*(\mathbf{y})$ is compact. The proof is completed. ∎

### G. Proof of Theorem 6.2

For any nonempty set $A \subseteq \mathbb{R}^n$, we define the following function:

$$d_A(x) \triangleq \inf_{y \in A} \|x - y\|$$

$d_A$ is called the distance function of $A$.

*Lemma G.1:* For any nonempty set $A$ contained in $\mathbb{R}^n$, the distance function $d_A : \mathbb{R}^n \to \mathbb{R}$ is Lipschitz continuous on $\mathbb{R}^n$ of rank 1.

*Proof:* Let $x_1, x_2 \in \mathbb{R}^n$. For any $y \in A$,

$$d_A(x_1) \leq \|x_1 - y\| \leq \|x_1 - x_2\| + \|x_2 - y\|.$$

Then, for any $y \in A$,

$$d_A(x_1) - \|x_1 - x_2\| \leq \|x_2 - y\|.$$

It follows that

$$d_A(x_1) - \|x_1 - x_2\| \leq \inf_{y \in A} \|x_2 - y\| = d_A(x_2).$$

This implies $d_A(x_1) - d_A(x_2) \leq \|x_1 - x_2\|$. Similarly, we can show that $d_A(x_2) - d_A(x_1) \leq \|x_1 - x_2\|$. Thus,

$$|d_A(x_1) - d_A(x_2)| \leq \|x_1 - x_2\|.$$

The proof is completed. ∎

*Lemma G.2:* For any nonempty set $A$ contained in $\mathbb{R}^n$, the signed distance function $\tilde{d}_A : \mathbb{R}^n \to \mathbb{R}$ is continuous.

*Proof:* According to Lemma G.1 and the definition of $\tilde{d}_A$, $\tilde{d}_A$ is continuous at every exterior point and every interior point of $A$. Let $x_0 \in \partial A$. For any $x \in A$,

$$\left|\tilde{d}_A(x) - \tilde{d}_A(x_0)\right| = \tilde{d}_A(x) = \inf_{y \in A} \|x - y\| \leq \|x - x_0\|.$$

Thus, $\tilde{d}_A$ is continuous at every boundary point of $A$. The proof is completed. ∎

*Proof:* [proof of Theorem 6.2] Let $x_1, x_2 \in \mathbb{R}^n$. If both $x_1$ and $x_2$ are in $A$ or both $x_1$ and $x_2$ are in $A^c$, then, by Lemma G.1,

$$\left|\tilde{d}_A(x_1) - \tilde{d}_A(x_2)\right| \leq \|x_1 - x_2\|.$$

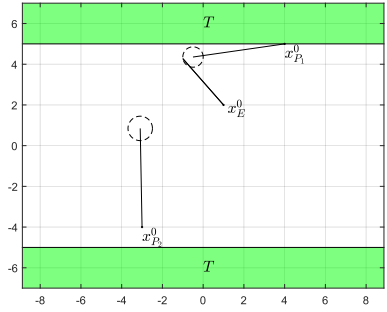

(a) S1: Both the pursuers and evader adopt the optimal strategies.

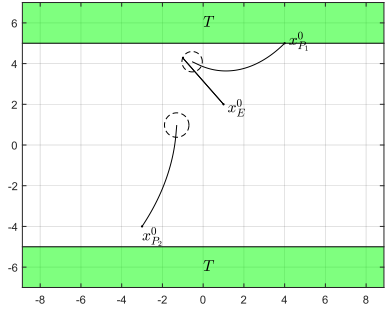

(b) S2: The evader plays optimally, while the pursuers adopt the pure pursuit strategies

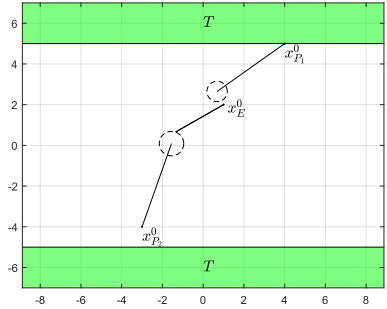

(c) S3: the pursuers play optimally, while the evader moves in an arbitrary direction.

Fig. 2: The trajectories of the players in the three scenarios

If $x_1 \in A$ and $x_2 \in A^c$, then $\tilde{d}_A(x_1) \leq 0$ and $\tilde{d}_A(x_2) \geq 0$. By Lemma G.2 and intermediate value theorem, there exists a point $x_3$ on the segment connecting $x_1$ and $x_2$ such that $\tilde{d}_A(x_3) = 0$, i.e. $x_3 \in \partial A$. Then,

$$\left|\tilde{d}_A(x_1) - \tilde{d}_A(x_2)\right| = \inf_{y \in A} \|x_1 - y\| + \inf_{y \in A} \|x_2 - y\| \quad \text{(G.11)}$$
$$\leq \|x_1 - x_3\| + \|x_2 - x_3\| = \|x_1 - x_2\|.$$

If $x_1 \in A^c$ and $x_2 \in A$, (G.11) also holds similarly. The proof is completed. ∎

### H. Numerical example

In this section, we present a numerical example. We select the following parameter values: $n = 2$, $m = 2$, $x_E^0 = (1, 2)^\top$, $x_{P_1}^0 = (4, 5)^\top$, $x_{P_2}^0 = (-3, -4)^\top$, $\alpha_1 = 1.5$, $\alpha_2 = 1.6$,

$l_1 = 0.5$, $l_2 = 0.6$. The sampling period for the numerical experiment is selected as $\Delta t = 0.01$. The target is the union of two half-planes: $T = \{(y_1, y_2) \in \mathbb{R}^2 : y_2 < -5 \vee y_2 > 5\}$. We set the terminal cost $g$ as $g(y_1, y_2) = 25 - y_2^2$.

Three different game scenarios (labeled S1–S3) are examined. In S1, both the pursuers and evader adopt the optimal strategies as in section V. In S2, the evader plays optimally, while the pursuers adopt the pure pursuit strategies: $\hat{\gamma}(x_{P_i}, x_E) = \frac{x_E - x_{P_i}}{\|x_E - x_{P_i}\|}$. In S3, only the pursuers play optimally, while the evader moves in an arbitrary direction. Figure 2 illustrates the resulting trajectories of the agents in each scenario. The change of the value over time can be found in Figure 3,

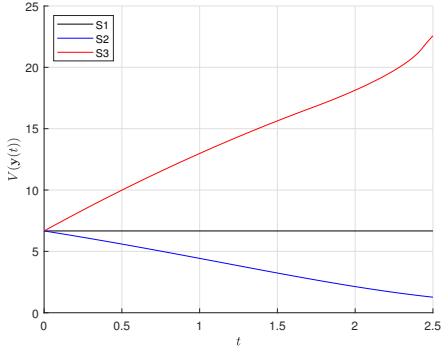

Fig. 3: The change of the value over time.



\end{document}